\documentclass[a4paper]{amsart}

\pdfoutput=1

\usepackage{stackrel}
\usepackage{amsmath}
\usepackage{amssymb}
\usepackage{amsthm}
\usepackage{mathtools}
\usepackage{commath}

\usepackage{booktabs}
\usepackage{microtype}
\usepackage{lmodern}
\usepackage[T1]{fontenc}
\usepackage[all]{xy}
\usepackage{dsfont}
\usepackage{xspace}
\usepackage{chngcntr}
\usepackage[english]{babel}
\usepackage[babel,english=american]{csquotes}
\usepackage{multirow}
\usepackage{enumerate}
\usepackage{mathrsfs}
\usepackage{stmaryrd}
\usepackage{tikz,pgffor}
\usetikzlibrary{arrows}

\usepackage{enumitem}
\usepackage{mathrsfs}

\usepackage{booktabs,multirow}
\usepackage{url}

\DeclareRobustCommand{\SkipTocEntry}[4]{}

\newtheorem{theorem}{Theorem}[section]

\newtheorem{lemma}[theorem]{Lemma}
\newtheorem{prop}[theorem]{Proposition}

\newtheorem{cor}[theorem]{Corollary}
\theoremstyle{definition}
\newtheorem{definition}[theorem]{Definition}
\newtheorem{remark}[theorem]{Remark}
\newtheorem{example}[theorem]{Example}

\numberwithin{equation}{section}

\newcommand{\minus}{-}

\DeclareMathOperator*{\Lie}{Lie}
\DeclareMathOperator*{\id}{id}

\DeclareMathOperator*{\Supp}{supp}
\DeclareMathOperator{\Div}{div}
\DeclareMathOperator{\rank}{rank}
\DeclareMathOperator{\conv}{conv}
\DeclareMathOperator{\II}{II}
\DeclareMathOperator{\III}{III}
\DeclareMathOperator{\IV}{IV}

\newcommand{\st}{\mathrm{st}}
\newcommand{\clst}{\overline{\mathrm{st}}}

\newcommand{\PGL}{\mathrm{PGL}}
\newcommand{\Spin}{\mathrm{Spin}}
\newcommand{\SL}{\mathrm{SL}}
\newcommand{\SO}{\mathrm{SO}}
\newcommand{\Cl}{\mathrm{Cl}}
\newcommand{\Pic}{\mathrm{Pic}}
\newcommand{\Sym}{\mathrm{Sym}}
\newcommand{\cone}{\mathrm{cone}}

\newcommand{\SHom}{\mathcal{H}om}
\newcommand{\Hom}{\mathrm{Hom}}

\newcommand{\relint}{\mathrm{relint}}
\newcommand{\topint}{\mathrm{int}}
\newcommand{\ohne}{\smallsetminus}

\newcommand{\Der}{\mathrm{Der}}

\newcommand{\widecheck}{\check}

\newcommand{\rleft}{\mathopen{}\mathclose\bgroup\left}
\newcommand{\rright}{\aftergroup\egroup\right}

\newcommand{\Ha}{\mathrm{\Gamma}}

\newcommand{\C}{\mathds{C}}
\newcommand{\Q}{\mathds{Q}}

\newcommand{\V}{\mathds{V}}
\newcommand{\Z}{\mathds{Z}}

\newcommand{\G}{\mathds{G}}
\newcommand{\Pb}{\mathds{P}}

\newcommand{\Om}{\mathcal{O}}

\newcommand{\Vm}{\mathcal{V}}
\newcommand{\Cm}{\mathcal{C}}
\newcommand{\Fm}{\mathcal{F}}
\newcommand{\Dm}{\mathcal{D}}

\newcommand{\Tm}{\mathscr{T}}

\newcommand{\Nm}{\mathcal{N}}
\newcommand{\Mm}{\mathcal{M}}

\newcommand{\Ff}{\mathfrak{F}}

\newcommand{\X}{\mathfrak{X}}

\newcommand{\F}{\mathfrak{F}}

\newcommand{\gf}{\mathfrak{g}}
\newcommand{\uf}{\mathfrak{u}}
\newcommand{\lf}{\mathfrak{l}}

\newcommand{\Ts}{\mathscr{T}}

\newcommand{\ie}{i.\,e.~}
\newcommand{\eg}{e.\,g.~}

\begin{document}
\selectlanguage{english}

\title{Gorenstein spherical Fano varieties}

\author{Giuliano Gagliardi}
\address{Fachbereich Mathematik, Universit\"at T\"ubingen, Auf der
Morgenstelle 10, 72076 T\"ubingen, Germany}
\curraddr{}
\email{giuliano.gagliardi@uni-tuebingen.de}
\thanks{}

\author{Johannes Hofscheier}
\address{Fachbereich Mathematik, Universit\"at T\"ubingen, Auf der
Morgenstelle 10, 72076 T\"ubingen, Germany}
\curraddr{}
\email{johannes.hofscheier@uni-tuebingen.de}
\thanks{}


\begin{abstract}
We obtain a combinatorial description of
Gorenstein spherical Fano varieties in terms of certain
polytopes, generalizing the combinatorial description of
Gorenstein toric Fano varieties by reflexive polytopes and its
extension to Gorenstein horospherical Fano varieties due to
Pasquier. Using this description,
we show that the rank of the Picard group of an arbitrary
$d$-dimensional $\Q$-factorial Gorenstein spherical Fano variety is
bounded by $2d$. This paper also contains an overview of the
description of the natural representative of the 
anticanonical divisor class of a spherical variety due to Brion.
\end{abstract}

\maketitle

\section{Introduction}
\label{sec:intr}

A complete complex algebraic variety is called Gorenstein Fano if it
is normal and its anticanonical divisor is Cartier and ample. For
toric varieties one has a nice characterization of Gorenstein Fano
varieties by convex geometry, namely one has a bijective
correspondence between Gorenstein toric Fano varieties and reflexive
polytopes, \ie lattice polytopes whose dual is a lattice polytope as
well (see \cite[Theorem 4.1.9]{Bat:DualPolyhedra}).  Generalizing the
notion of a reflexive polytope, Pasquier established a similar
correspondence for horospherical varieties in
\cite{Pasquier:FanoHorospherical}. In this paper, we extend this
correspondence to (arbitrary) spherical varieties.

In \cite{Brion:cc}, Brion determined a natural representative of the
anticanonical divisor class for any spherical variety $X$, on which
our combinatorial description of Gorenstein spherical Fano varieties
essentially depends. For this reason, we have found it useful to give
a self-contained overview of Brion's as well as some related results
of Luna, together with some motivation.
The new results of this paper, \ie
Theorem~\ref{thm:bijection-Fano-poly} and Theorem~\ref{theorem:le-2d},
are stated afterwards.

Let $G$ be a connected reductive complex algebraic group and $B
\subseteq G$ a Borel subgroup.  A closed subgroup $H \subseteq G$ is
called \emph{spherical} provided that the homogeneous space $G/H$ contains
an open $B$-orbit. Let $H \subseteq G$ be a spherical subgroup. A
$G$-equivariant open embedding $G/H \hookrightarrow X$ into a normal
irreducible $G$-variety $X$ is called a \emph{spherical embedding}, and
$X$ is called a \emph{spherical variety}.

An anticanonical divisor is a Weil divisor $-K_X$ of $X$ such that
$\Om_X(-K_X) = \check{\omega}_X$ where $\check{\omega}_X$ is the
anticanonical sheaf of $X$, \ie
the reflexive sheaf coinciding with the top exterior power of
the tangent sheaf on the smooth locus of $X$. 
It is equipped with a natural $G$-linearization
(see Section~\ref{sec:CoTangentSheafGVariety}).

It is known that the open $B$-orbit $U$ is isomorphic to
$(\C^*)^r\times\C^s$ (see~\cite[Theorem~5]{ros63}) and hence has
trivial divisor class group.  In particular, the invertible sheaf
$\widecheck\omega_{G/H}$ is trivial on $U$, and we might ask whether
there is a natural choice of a generator $s \in \Ha(U,
\widecheck\omega_{G/H})$.

We first recall the well-known case where $G = B = T$ is a torus and
$H$ is trivial (see, for instance, \cite[Section~4.3]{fulttor}).  A
spherical embedding $G/H \hookrightarrow X$ is then simply a toric
variety $X$ with embedded torus $U = G/H \cong T$, and it is possible
to show that there is a unique generator $s \in \Ha(U,
\widecheck\omega_{G/H})$ which is $T$-invariant. This generator can be
explicitly written as
\begin{align*}
  s = x_1 \frac{\partial}{\partial x_1} \wedge \ldots \wedge x_n
  \frac{\partial}{\partial x_n}
\end{align*}
where $x_1,\ldots,x_n$ is a choice of coordinates for the torus
$T$. In addition to being $T$-invariant, this section has the
important property that it has a zero of order 1 along every
$T$-invariant divisor in $X$.

In general, however, there need not be any $B$-invariant section $s
\in \Ha(U,\widecheck\omega_{G/H})$ as the following example shows.

\begin{example}
  \label{exa:no_invariant_generators}
  Let $G=\SL_2$ and $H \subseteq G$ a Borel subgroup.  Then the open
  $B$-orbit $U$ in $G/H$ is isomorphic to $\C$. It is not difficult to
  see that there is only one $B$-semi-invariant section in
  $\Ha(U,\widecheck\omega_{G/H})$ and that it is not $B$-invariant.
\end{example}

As we cannot expect a $B$-invariant section, we have to look for
something else.  An equivalent characterization for a homogeneous
space $G/H$ to be spherical is that the $G$-module $\Ha(G/H,\mathscr{L})$ is
multiplicity-free for every $G$-linearized invertible sheaf
$\mathscr{L}$, \ie the multiplicity of any simple $G$-module in the
decomposition of the module of global sections is at most $1$.  We
obtain
\begin{align*}
  \Ha(G/H, \widecheck\omega_{G/H}) \cong \bigoplus V_\chi\text{,}
\end{align*}
where $\chi$ runs over pairwise different dominant weights of $B$ and
$V_\chi$ is the simple $G$-module of highest weight $\chi$.

This decomposition is very simple when $H \subseteq G$ is a parabolic
subgroup since then $\Ha(G/H, \widecheck\omega_{G/H})$ is a simple
$G$-module by the Borel-Weil-Bott theorem (see
\cite{Demazure:Bott,Demazure:BottSimple}).
In particular, exactly one dominant
weight occurs (compare this with
Example~\ref{exa:no_invariant_generators}), and there is a unique
choice (up to a constant factor) of a $B$-semi-invariant section $s
\in\Ha(G/H,\widecheck\omega_{G/H})$ which restricts to a generator $s
\in \Ha(U, \widecheck\omega_{G/H})$.

Let $G/H$ again be an arbitrary spherical homogeneous space.  We
denote by $P \subseteq G$ the stabilizer of the open $B$-orbit $U$.
If $H$ contains a maximal unipotent subgroup of $G$, the homogeneous
space $G/H$ is called \emph{horospherical}, and the normalizer of $H$
in $G$ is a parabolic subgroup conjugated to the opposite parabolic of
$P$, which we denote by $P^{\minus}$.  Hence there is a natural
morphism $\pi \colon G/H \to G/P^{\minus}$, which is known to be a
torus fibration. Therefore $\pi^*( \widecheck \omega_{G/P^{\minus}} )
= \widecheck \omega_{G/H}$, the simple $G$-module $\Ha( G/P^{\minus},
\widecheck \omega_{G/P^{\minus}} )$ is a direct summand of $\Ha( G/H,
\widecheck \omega_{G/H} )$, and a unique $B$-semi-invariant section $s
\in \Gamma( G/P^{\minus}, \widecheck \omega_{G/P^{\minus}} ) \subseteq
\Ha( G/H, \widecheck\omega_{G/H} )$ exists, whose weight we denote by
$\kappa_P \in \X(B)$.

For arbitrary spherical homogeneous spaces $G/H$, there is no natural
morphism $\pi \colon G/H \to G/P^{\minus}$, but the following
statement is nevertheless valid.

\begin{theorem}[{\cite[4.1 and 4.2]{Brion:cc}}]
  \label{th:a1}
  The simple $G$-module
  $\Ha(G/P^{\minus},\widecheck\omega_{G/P^{\minus}})$ is a direct
  summand of $\Ha(G/H, \widecheck\omega_{G/H})$.  Equivalently, there
  exists a $B$-semi-invariant section
  \begin{align*}
    s \in \Ha(G/H, \widecheck\omega_{G/H})
  \end{align*}
  of weight $\kappa_P$, which restricts to a generator $s \in \Ha(U,
  \widecheck\omega_{G/H})$. For any spherical embedding $G/H
  \hookrightarrow X$ the section $s$ extends to a global section on
  $X$, and we have
  \begin{align*}
    \Div s = \sum_{i=1}^k m_i D_i + \sum_{j=1}^n X_j
  \end{align*}
  where $D_1, \ldots, D_k$ are the $B$-invariant prime divisors
  in $G/H$ (identified with their closures in $X$),
  the $m_i$ are positive integers depending only on the
  homogeneous space $G/H$, and $X_1, \ldots, X_n$ are the $G$-invariant
  prime divisors in $X$.
\end{theorem}

\begin{remark}
The weight $\kappa_P$ is the weight $\pi(\Div s)$
in the sense of \cite[3.3]{l-brion-pic}.
\end{remark}

The fact from Theorem~\ref{th:a1} that the section $s$ has a zero of
order 1 along any $G$-invariant prime divisor actually characterizes
this section.  This is a straightforward application of basic facts
from the embedding theory of spherical homogeneous spaces.

\begin{theorem}
  \label{th:a2}
  Let $s \in \Gamma( U, \widecheck \omega_{G/H} )$ be a generator
  (which is automatically $B$-semi-invariant). Then the following
  conditions are equivalent:
  \begin{enumerate}
  \item The section $s$ is of $B$-weight $\kappa_P$.
  \item For any spherical embedding $G/H \hookrightarrow X$ the
    section $s$ has a zero of order 1 along any $G$-invariant prime
    divisor.
  \end{enumerate}
\end{theorem}

We denote by $\Dm = \{D_1,
\ldots, D_k\}$ the set of $B$-invariant prime divisors in $G/H$, whose
elements are called the \emph{colors} of $G/H$.
In order to obtain an explicit formula for the coefficients $m_i$ in
Theorem~\ref{th:a1}, Brion divided the colors into several types
(see~\cite[4.2]{Brion:cc}). These types are in agreement with the
definition of the types of colors due to Luna (see~\cite[2.7]{Luna:cc},
\cite[2.3]{Luna:typea}, see also \cite[Section~30.10]{ti}), which we
now explain. For additional information on the types of colors, we
refer the reader to Section~\ref{sec:typcol}.

We choose a maximal torus $T \subseteq B$, denote by $R \subseteq
\X(T) = \X(B)$ the associated root system, and write $S \subseteq R$
for the set of simple roots corresponding to $B$.  For $\alpha \in S$
we denote by $P_\alpha \subseteq G$ the corresponding minimal 
parabolic subgroup containing $B$, and define
\begin{align*}
  \Dm(\alpha) \coloneqq \{D_i \in \Dm : P_\alpha \cdot D_i \ne
  D_i\}\text{.}
\end{align*}
As the colors are not $G$-stable, every color is moved by at least
one minimal parabolic subgroup, so that we have $\Dm =
\bigcup_{\alpha \in S}\Dm(\alpha)$. Moreover, the stabilizer $P \subseteq G$
of the open $B$-orbit is the parabolic subgroup containing $B$
corresponding to the set $S^p \coloneqq \{\alpha \in S : \Dm(\alpha) =
\emptyset\}$.

We denote by $\Mm \subseteq \X(B)$ the weight lattice of
$B$-semi-invariants in the function field $\C(G/H)$ and by $\Nm
\coloneqq \Hom(\Mm, \Z)$ the dual lattice together with the natural
pairing $\langle \cdot, \cdot\rangle \colon \Nm \times \Mm \to \Z$.
We denote by $\Vm$ the set of $G$-invariant discrete valuations on
$\C(G/H)$, and define the map $\iota \colon \Vm \to \Nm$ by $\langle
\iota(\nu), \chi \rangle \coloneqq \nu(f_\chi)$ where $f_\chi \in
\C(G/H)$ is $B$-semi-invariant of weight $\chi \in \Mm$ and unique up
to a constant factor.  As the map $\iota$ is injective, we may
consider $\Vm$ as a subset of the vector space $\Nm_\Q \coloneqq \Nm
\otimes_\Z \Q$. It is known that $\Vm$ is a cosimplicial cone
(see~\cite{brg}), called the \emph{valuation cone} of $G/H$. In
particular, the valuation cone is full-dimensional.  By $\Sigma$ we
denote the set of primitive generators in $\Mm$ of the extremal rays
of the negative of the dual of the valuation cone $\Vm$. The elements
in $\Sigma$ are called the \emph{spherical roots} of $G/H$.

The type of a color $D_i \in \Dm(\alpha)$ is defined as follows: If
$\alpha \in \Sigma$, we say that $D_i$ is of type $a$. If $2\alpha \in
\Sigma$, we say that $D_i$ is of type $2a$. Otherwise, we say that
$D_i$ is of type $b$. The type does not depend on the choice $\alpha
\in S$ such that $D_i \in \Dm(\alpha)$. Moreover,
we have $|\Dm(\alpha)| \le 2$ with $|\Dm(\alpha)| = 2$
if and only if $\alpha \in \Sigma$ (\ie the colors
in $\Dm(\alpha)$ are of type $a$).

We can now state the explicit formula for the coefficients $m_i$ in
the expression for $\Div s$ due to Brion and Luna.

\begin{theorem}[{\cite[Theorem~4.2]{Brion:cc}, \cite[3.6]{Luna:cc}}]
  \label{th:ac}
  We have
  \begin{align*}
    m_i &= \tfrac{1}{2} \langle \alpha^\vee, \kappa_P \rangle = 1
    &&\text{ for $D_i$ of type $a$ or $2a$,}\\
    m_i &= \langle \alpha^\vee, \kappa_P \rangle \ge 2 && \text{ for $D_i$
      of type $b$.}
  \end{align*}
\end{theorem}

This concludes the overview, so that we are now able to explain the
combinatorial description of Gorenstein spherical Fano varieties. Let
$m_1, \ldots, m_k \in \Z_{>0}$ be the coefficients of the colors in
the expression for the anticanonical divisor from Theorem~\ref{th:a1}.
We define the map $\rho \colon \Dm \to \Nm$ by $\langle \rho(D_i),
\chi\rangle \coloneqq \nu_{D_i}(f_\chi)$ where $\nu_{D_i}$ is the discrete
valuation associated to $D_i \in \Dm$.

\begin{definition}
  \label{defn:Q_X}
  Let $G/H \hookrightarrow X$ be a complete spherical embedding, and
  let $X_1,\ldots, X_n$ be the $G$-invariant prime divisors in $X$. We
  define the polytope
  \begin{align*}
    Q_X \coloneqq \conv \rleft( \frac{ \rho(D_1) }{ m_1}, \ldots,
    \frac{ \rho(D_k) }{ m_k }, \nu_{X_1}, \ldots, \nu_{X_n} \rright)
    \subseteq \Nm_\Q.
  \end{align*}
\end{definition}

\begin{remark}
  The polytope introduced in Definition~\ref{defn:Q_X} has been used
  by Alexeev and Brion to prove boundedness of spherical Fano
  varieties in \cite{albr}.
\end{remark}

The generalization of the notion of a reflexive polytope to the theory
of spherical varieties is the following (it is a generalization of
\cite[D\'{e}finition~3.3]{Pasquier:FanoHorospherical}).  For a
polytope $Q$ we denote by $Q^*$ its dual polytope, and for a face $F
\preceq Q$ we denote by $\widehat{F} \preceq Q^*$ its dual face.

\begin{definition}
  \label{def:qghrefl}
  A polytope $Q \subseteq \Nm_\Q$ is called \emph{$G/H$-reflexive} if
  the following conditions are satisfied:
  \begin{enumerate}
  \item $\rho( D_i ) / m_i \in Q$ for every $i = 1, \ldots, k$.
  \item $0 \in \topint(Q)$.
  \item Every vertex of $Q$ is contained in $\{ \rho( D_i ) / m_i : i
    = 1, \ldots, k \}$ or $\Nm \cap \Vm$.
  \item Every vertex $v \in Q^*$ satisfying $\relint( \cone(
    \widehat{v} ) ) \cap \Vm \neq \emptyset$ lies in the lattice
    $\Mm$.
  \end{enumerate}
  Note that $\cone(\widehat{v})$ is a full-dimensional cone in
  $\Nm_\Q$ for every vertex $v \in Q^*$.
\end{definition}

\begin{theorem}
  \label{thm:bijection-Fano-poly}
  The assignment $X \mapsto Q_X$ induces a bijection between
  isomorphism classes of Gorenstein spherical Fano embeddings $G/H
  \hookrightarrow X$ and $G/H$-reflexive polytopes.
\end{theorem}

Using Theorem \ref{thm:bijection-Fano-poly}, one may translate
questions about Gorenstein spherical Fano varieties into the realm of
convex combinatorics. Applying this approach, we are going to prove
the following bound on the Picard number.
\begin{theorem}
  \label{theorem:le-2d}
  Let $X$ be a $\Q$-factorial Gorenstein spherical Fano variety of
  dimension $d$ and Picard number $\rho_X$.  Then we have
  \begin{align*}
    \rho_X \le 2d\text{,}
  \end{align*}
  with $\rho_X = 2d$ if and only if $d$ is even and $X \cong
  (S_3)^{d/2}$ where $S_3$ is the blowup of $\Pb^2$ at three
  non-collinear points.
\end{theorem}

Theorem~\ref{theorem:le-2d} has been proven by Casagrande in the case
of a toric variety $X$ (see \cite{Casagrande:numVerts}) and by
Pasquier in the case of a horospherical variety $X$ (see
\cite{Pasquier:FanoHorospherical}).  Our proof is inspired by the two
previous works.  Observe that Theorem~\ref{theorem:le-2d} does not
hold for an arbitrary variety $X$, \eg if $S$ is the surface given by
blowing-up $\Pb^2$ in eight general points, the variety $S^m$ has
Picard number $9m$ (see \cite{Debarre:FanoVarieties}).

\subsection*{List of general notation}
\renewcommand{\descriptionlabel}[1]{\hspace\labelsep #1}
\begin{description}[leftmargin=8em,style=nextline]
\item[{$\X(G)$}] character lattice of a connected algebraic group $G$,
\item[{$Q^*$}] dual polytope to a polytope $Q$ in a vector space $V$, \ie \\
  $Q^* = \{ v \in V^* : \langle u, v \rangle \ge -1 \text{ for every
    $u \in Q$}\}$,
\item[{$\widehat{F}$}] dual face to a face $F$ of a polytope $Q$, \ie \\
  $\widehat{F} \coloneqq \{ v \in Q^* : \langle u, v \rangle = -1
  \text{ for every $u \in F$}\}$,
\item[{$\topint(A)$}] topological interior of a subset $A$ in some
  finite-dimensional vector space,
\item[{$\relint(A)$}] relative interior of a subset $A$ in some
  finite-dimensional vector space, \ie topological interior of $A$ in
  the affine span of $A$.
\end{description}

\tableofcontents

\section{Notation and generalities}
\label{sec:notat-gener}

Spherical embeddings admit a combinatorial description due to the
Luna-Vust theory (see~\cite{lunavust, knopsph}). 
Similarly to the theory of toric
varieties, one obtains a description of spherical embeddings of $G/H$
by colored fans, which are combinatorial objects living in the vector
space $\Nm_\Q$.

\begin{definition}
  A \emph{colored cone} is a pair $(\Cm, \Fm)$ where $\Fm \subseteq
  \Dm$ and $\Cm \subseteq \Nm_\Q$ is a cone generated by $\rho(\Fm)$
  and finitely many elements of $\Vm$.  A colored cone is called
  \emph{supported} if $\relint(\Cm) \cap \Vm \ne \emptyset$.  A
  colored cone is called \emph{strictly convex} if $\Cm$ is strictly
  convex and $0 \notin \rho(\Fm)$.
\end{definition}

\begin{definition}
  A \emph{face} of a colored cone $(\Cm, \Fm)$ is a colored cone
  $(\Cm', \Fm')$ such that $\Cm'$ is a face of $\Cm$ and $\Fm' = \Fm
  \cap \rho^{-1}(\Cm')$. It is called a \emph{supported face} if it is
  supported as a colored cone.
\end{definition}

\begin{definition}
  A \emph{colored fan} is a nonempty finite collection $\Ff$ of
  strictly convex colored cones such that for every $(\Cm, \Fm) \in
  \Ff$ every face of $(\Cm, \Fm)$ is also in $\Ff$ and for every $v
  \in \Nm_\Q$ there is at most one $(\Cm, \Fm) \in \F$ with $v \in
  \relint(\Cm)$.  A colored fan $\Ff$ is called \emph{complete} if
  $\Supp \Ff \coloneqq \bigcup_{(\Cm, \Fm) \in \Ff} \Cm = \Nm_\Q$.
\end{definition}

\begin{definition}
  A \emph{supported colored fan} is a nonempty finite collection $\Ff$
  of strictly convex supported colored cones such that for every
  $(\Cm, \Fm) \in \Ff$ every supported face of $(\Cm, \Fm)$ is also in
  $\Ff$ and for every $v \in \Vm$ there is at most one $(\Cm, \Fm) \in
  \F$ with $v \in \relint(\Cm)$.  A supported colored fan $\Ff$ is
  called \emph{complete} if $\Supp \Ff \supseteq \Vm$.
\end{definition}

\begin{remark}
  We have defined the terms \enquote{colored cone} and
  \enquote{colored fan} with and without the adjective
  \enquote{supported}.  In the literature, only the supported versions
  are usually defined (and without the adjective \enquote{supported}).
\end{remark}

\begin{remark}
  \label{rem:s_maps_complete_to_complete}
  There is a natural map
  \begin{align*}
    \{ \text{colored fans} \} &\to \{ \text{supported colored
      fans} \}\\
    \Ff &\mapsto \Ff_{\Supp} \coloneqq \{ \sigma \in \Ff : \sigma
    \text{ is supported}\}\text{.}
  \end{align*}
  If $\Ff$ is complete, every point in $\Vm$ is contained in
  $\relint(\Cm)$ for exactly one $(\Cm, \Fm) \in \Ff$, which is
  supported, so $\Ff_{\Supp}$ is complete as well (as a supported
  colored fan).
\end{remark}

\begin{theorem}[{\cite[Theorem~3.3]{knopsph}}]
  Supported colored fans are in bijective correspondence with
  isomorphism classes of spherical embeddings $G/H \hookrightarrow
  X$. Moreover, $X$ is complete if and only if the corresponding
  supported colored fan is complete.
\end{theorem}

We now recall some results on divisors in spherical varieties due to
Brion (see~\cite{l-brion-pic}).  We use \cite[Section~17]{ti} as
general reference. Let $G/H \hookrightarrow X$ be a complete spherical
embedding and $\Ff$ the corresponding supported colored fan. Let $E$
be a $B$-stable Weil divisor.  The Weil divisor $E$ is $\Q$-Cartier
(resp. Cartier) if and only if for every maximal supported colored
cone $(\Cm, \Fm) \in \Ff$ there exists $v_\Cm \in \Mm_\Q$
(resp. $v_\Cm \in \Mm$) such that for every $B$-stable prime divisor
$D$ which contains the $G$-orbit in $X$ corresponding to $(\Cm, \Fm)$
the multiplicity of $D$ in $E$ is $\langle \rho(D), v_\Cm \rangle$ if
$D$ is a color and $\langle \nu_D, v_\Cm \rangle$ if $D$ is
$G$-stable. The linear functions $\langle \cdot, v_\Cm \rangle : \Cm
\cap \Vm \subseteq \Nm_\Q \to \Q$ may be pasted together to a
piecewise linear function $\psi_E \colon \Vm \to \Q$.

\begin{prop}[{\cite[Proposition~3.1]{l-brion-pic}}]
\label{prop:sphfqf}
$X$ is $\Q$-factorial (resp.~locally factorial) if and only if
for every maximal $(\Cm, \Fm) \in \Ff$
the cone $\Cm$ is spanned by a part of a $\Q$-basis of $\Nm_\Q$
(resp. by a part of a $\Z$-basis of $\Nm$)
containing $\rho(\Fm)$ and $\rho|_{\Fm}$
is injective.
\end{prop}

\begin{prop}[{\cite[Th\'eor\`eme~3.3]{l-brion-pic}, see also \cite[Corollary~17.24]{ti}}]
  \label{prop:ample}
  Let $E$ be a $\Q$-Cartier divisor on the complete spherical variety
  $X$ and $\psi_E\colon \Vm \to \Q$ the associated piecewise linear
  function.  Then $E$ is ample if and only if
  \begin{enumerate}
  \item the piecewise linear function $\psi_E$ is strictly convex, \ie
    for every $u \in \Cm \setminus \Cm'$ we have $\langle u, v_\Cm
    \rangle > \langle u, v_{\Cm'}\rangle$ for any two maximal $(\Cm,
    \Fm), (\Cm', \Fm') \in \Ff$, and
  \item for every maximal $(\Cm, \Fm) \in \Ff$ and every $D \in
    \Dm\setminus \Fm$ the multiplicity of $D$ in $E$ is strictly
    greater than $\langle \rho(D), v_\Cm\rangle$.
  \end{enumerate}
\end{prop}

Finally, we will require the following definition.

\begin{definition}
  \label{def:colored_face_fan}
  Let $Q \subseteq \Nm_\Q$ be a polytope with $0 \in \topint(Q)$.
  We associate to it the colored fan 
  $\Ff(Q)$ 
  consisting of the colored cones $( \cone(F), \rho^{-1}(F) )$ for all
  proper faces $F$ of $Q$.  As $Q$ is full-dimensional, the colored
  fan $\Ff(Q)$ is complete. The colored fan $\Ff(Q)$ is called
  the \emph{colored face fan} of $Q$.
  We write $\Ff_{\Supp}(Q)$ for $(\Ff(Q))_{\Supp}$. 
\end{definition}

\section{The co- and tangent sheaves of a smooth $G$-variety}
\label{sec:CoTangentSheafGVariety}
In this section, let $X$ be an arbitrary smooth $G$-variety for an
arbitrary algebraic group $G$. Then $G$ acts on $X$ by an action
morphism $\alpha \colon G \times X \to X$.  We denote by $\mu \colon G
\times G \to G$ the multiplication morphism of the algebraic group $G$
and by $\pi_X \colon G \times X \to X$ (resp.~
$\pi_{G\times X} \colon G \times G \times X \to G \times X$)
the natural projection on the second (resp.~on the second and third)
factor. Let us
repeat the definition of a $G$-linearization of a quasicoherent sheaf
(see \cite[Definition~C.2]{ti} or
\cite[Definition~1.6]{MumfordFogartyKirwan:GIT}).

\begin{definition}
  \label{defn:GSheaf}
  A \emph{$G$-linearization} of a quasicoherent sheaf $\F$ on $X$ is
  an isomorphism of quasicoherent sheaves $\widehat{\alpha} \colon
  \pi_X^*\F \xrightarrow{\sim} \alpha^*\F$ satisfying the cocycle
  condition, \ie the diagram in Figure~\ref{fig:ccc} commutes.
  \begin{figure}[!ht]
    \xymatrix@C=1.5cm{
      \rleft(\pi_X \circ\pi_{G \times X} \rright)^*\F
      \ar[r]^{\pi_{G \times X}^* \widehat{\alpha}} \ar@{=}[dd] &
      \rleft(\alpha\circ\pi_{G\times X}\rright)^*\F\ar@{=}[d]& \\
      &\rleft( \pi_X \circ( \id_G \times \alpha ) \rright)^*\F
      \ar[r]^{\rleft( \id_G \times \alpha \rright)^* \widehat{\alpha}}
      & \rleft(\alpha\circ(\id_G\times\alpha)\rright)^*\F\ar@{=}[d] \\
      \rleft( \pi_X \circ( \mu \times \id_X ) \rright)^*\F
      \ar[rr]^{\rleft( \mu \times \id_X \rright)^* \widehat{\alpha}} &
      & \rleft( \alpha \circ( \mu \times \id_X ) \rright)^*\F
    }
    \caption{The cocycle condition.}
    \label{fig:ccc}
  \end{figure}
\end{definition}

Recall that the \emph{cotangent sheaf $\Omega_X$} of $X$ is, locally
on affine open neighbourhoods $U$, given as the sheaf associated to
the module of K\"ahler differentials
$\Omega_{\mathcal{O}_X(U)/\C}$. The pullback of differential forms
with respect to the action morphism $\alpha$ yields the inverse of a
$G$-linearization of the cotangent sheaf, namely
$\widehat{\alpha}^{-1} \colon \alpha^*\Omega_X \xrightarrow{\sim}
\pi_X^*\Omega_X$. As $X$ is smooth, we may dualize
$\widehat{\alpha}^{-1}$ and obtain a $G$-linearization of the tangent
sheaf $\Tm_X \coloneqq \SHom_{\Om_X}(\Omega_X,\Om_X)$, namely
$\widehat{\beta} \coloneqq ( \widehat{\alpha}^{-1})^\vee \colon
\pi_X^*\Tm_X \xrightarrow{\sim} \alpha^*\Tm_X$.

For an affine open subset $U \subseteq X$ and $g\in G$, let us denote
the coordinate rings of $U$ and $g\cdot U$ by $A$ and $B$
respectively.  The element $g\in G$ acts on a local section
$\delta\in\Der_\C(A,A)=\Gamma(U,\Tm_X)$ by restricting the
$G$-linearization to $\{g\}\times X$, \ie $g \cdot \delta
\coloneqq\widehat{\beta} |_{\{g\} \times X} (\delta)$. It is
straightforward to check that
\[
g\cdot\delta=\lambda_g^\#\circ\delta\circ\lambda_{g^{-1}}^\#\in\Der_\C(B,B)
= \Gamma(g \cdot U,\Tm_X)\text{,}
\]
where $\lambda_g \colon X \to X$ is given by $x \mapsto g^{-1}\cdot
x$.

Let $\G=(\C,+)$ or $\G=(\C^*,\cdot)$ be a one-dimensional connected
algebraic group with neutral element $e \in \G$.  We will recall how
one can associate a global vector field $\uf\in\Gamma(X,\Tm_X)$ to a
one-parameter subgroup $u \colon \G \to G$. The coordinate ring of
$\G$ is either the polynomial ring $\C[t]$ or the Laurent polynomial
ring $\C[t^{\pm1}]$. In both cases, we have a natural choice of a
basis of the tangent space of $\G$ over the point $e$, namely
$\tfrac{\partial}{\partial t}\big|_{e} \in \Tm_\G|_e$.  Let $U$ be an
affine open subset of $X$ with coordinate ring $A$. The restriction of
$\uf$ to $U$ lies in $\Der_\C(A,A)$ and is given by
\[
\rleft(\uf|_Uf\rright)(x)\coloneqq \frac{\partial}{\partial t}\bigg|_e
f\rleft(u(t)\cdot x\rright)
\]
for $f\in A$, $x\in U$.  It is straightforward to check that these
local sections are well-defined and glue to a global section
$\uf\in\Gamma(X,\Tm_X)$.

\begin{remark}
  \label{rem:Invariant_vector_field}
  Assume that $X=G$ and $G$ acts on itself by left translation.  It is
  then straightforward to check that $\uf$ is an \emph{invariant}
  vector field, \ie $\rho_g^\#\circ\uf\circ\rho_{g^{-1}}^\#=\uf$ for
  all $g\in G$, where $\rho_g \colon G \to G$ is given by $h \mapsto
  hg$. Moreover, $\uf|_e \in\Lie u(\G)$.
\end{remark}

\section{The anticanonical sheaf of a spherical variety}
\label{sec:antic-sheaf-spher}

From now on, we continue to use the notation and the assumptions
from the introduction. In this section, we reproduce (with more detail)
the proof of \cite[4.1]{Brion:cc}.

Let $G/H \hookrightarrow X$ be a spherical embedding.  We denote the
$G$-invariant prime divisors in $X$ by $X_1, \ldots, X_n$.  
As we are interested in the anticanonical sheaf of $X$, we may
assume that $X$ does not contain $G$-orbits of codimension two or
greater.  In particular, $X$ is smooth toroidal and $\Pic(X) = \Cl(X)$.

Since $X$ is a smooth toroidal variety, it has the following local
structure (see \cite[Theorem~29.1]{ti}): The set $X^0\coloneqq
X\ohne\bigcup_{i=1}^k\overline{D_i}$ is stable by $P$. Let $L$ be the Levi subgroup of $P$ containing $T$. There exists
a closed $L$-stable subvariety $Z$ of $X^0$
such that
\begin{align*}
  R_u(P)\times Z&\to X^0 \\
  (u,z)&\mapsto u\cdot z
\end{align*}
is a $P$-equivariant isomorphism. The kernel of the $L$-action on $Z$,
which we denote by $L_0$, contains $(L,L)$ and $Z$ is a toric
embedding of $L/L_0$. Every $G$-orbit intersects $Z$ in a unique
$L/L_0$-orbit.

Any $x_0 \in U$ corresponds to some
$(u_0,z_0)\in R_u(P)\times Z$ under the isomorphism above.
We fix $x_0$ such that $u_0$ is the neutral element of $R_u(P)$.
By replacing $H$ with a conjugate, we may assume that $x_0$ has stabilizer
$H \subseteq G$.

We denote by $\widecheck\omega_X \coloneqq\bigwedge^{\dim X}\Tm_X$,
\ie the top exterior power of the tangent sheaf, the anticanonical
sheaf of $X$.  It is an invertible sheaf and carries a natural
$G$-linearization induced from the $G$-linearization of $\Tm_X$.

The variety $Z$ is a toric variety for a quotient torus
of $T$. Let $T_0$ be the kernel of the $T$-action on $Z$ and let $T_1$
be a subtorus of $T$ with $T=T_0T_1$ such that $T_0\cap T_1$ is
finite. We have a commutative diagram of equivariant morphisms with
respect to the action of $B_0 \coloneqq R_u(P)T_1$:
\[
\xymatrix{
  B_0\ar[rr]\ar[dr] && B\cdot x_0\\
  &B_0/(T_0\cap T_1)\ar@{=}[ur]& }
\]
The arrows are finite coverings. In particular, the tangent space of
$X$ at $(u_0,z_0)$ is isomorphic to the direct sum of the tangent
spaces of $R_u(P)$ and $T_1$ at the corresponding neutral
elements. This is the lie algebra of $B_0$ which decomposes as
\[
\Lie B_0 =\Lie T_1 \oplus\bigoplus_{\alpha\in
  R^+\setminus\left<S^p\right>}\gf_\alpha
\]
where $\left<S^p\right>$ denotes the root system generated by $S^p$
and $\gf_\alpha$ denotes the subspace of $T$-semi-invariant vectors of
weight $\alpha$ in $\gf \coloneqq \Lie G$. Choose a realization
$(u_\alpha)_{\alpha\in R}$ of the root system $R$ (see
\cite[\S8.1]{Springer:LAG}) and choose a basis
$\lambda_1,\ldots,\lambda_r$ of the lattice of one-parameter
multiplicative subgroups of $T_1$.
In Section~\ref{sec:CoTangentSheafGVariety} we have seen how to
associate a global vector field $\uf_\alpha$ (resp.~$\lf_1, \ldots,
\lf_r$) to the one parameter subgroup $u_\alpha$ (resp.~$\lambda_1,
\ldots, \lambda_r$). We obtain a global section
\[
s\coloneqq\rleft(\bigwedge_{\alpha\in
  R^+\setminus\left<S^p\right>}\uf_\alpha
\rright)\wedge\lf_1\wedge\ldots\wedge\lf_r\in\Gamma(X,\widecheck\omega_X).
\]

\begin{prop}[{\cite[Proposition 4.1]{Brion:cc}}]
  The zero set of $s$ is exactly the union of the closures of the
  colors $D_i$ and the boundary divisors $X_j$.
\end{prop}
\begin{proof}
  Let $u$ be a one-parameter subgroup used in the definition of $s$,
  \ie $u=u_\alpha$ or $u \in \{\lambda_1,\ldots,\lambda_r\}$.  Let $Y$
  be a $B$-stable divisor, \ie $Y=\overline{D_i}$ or $Y=X_j$. Since $u$ maps into
  $B$, by construction, $\uf|_y \in \Tm_Y|_y$ for all $y\in Y$. Since
  $Y$ has codimension $1$ in $X$ and the number of vector fields which
  have been wedged to obtain $s$ is $\dim X$, the global section $s$
  vanishes on $Y$.
 
  We now show that $s$ vanishes nowhere on the open $B$-orbit. Since
  $u$ maps into $B_0$, we may define a global vector field $\uf'$ on
  $B_0$. By the local structure theorem, $B_0$ is a finite covering of
  $B\cdot x_0$.  In particular, we have a well-defined pushforward of
  vector fields, and $\uf|_{B\cdot x_0}$ is the pushforward of $\uf'$.
  Since $\uf'$ is $T_0 \cap T_1$-invariant
  (see~Remark~\ref{rem:Invariant_vector_field}), it suffices to show
  that $s' \in \Gamma(B_0, \widecheck\omega_{B_0})$, \ie the global
  section of the anticanonical sheaf of $B_0$ which arises by wedging
  all the global vector fields $\uf'$, vanishes nowhere.
 
  Since $\uf'|_e \in \Lie u(\G)$, it follows that $s'|_e$ arises by
  wedging a basis of $\Lie B_0$. In particular, $s'$ does not vanish
  at $e$. Since $s'$ is invariant, it follows that it vanishes
  nowhere.
\end{proof}
Recall from the introduction that, by the Borel-Weil-Bott theorem, the
space $\Gamma(G/P^{\minus},\widecheck\omega_{G/P^{\minus}})$ is a
simple $G$-module, whose highest weight we denote by $\kappa_P$.
Then the following result implies Theorem~\ref{th:a1}.

\begin{prop}[{\cite[Proof of Theorem~4.2]{Brion:cc}}]
  \label{prop:weight_of_special_section}
  The global section $s \in \Gamma(X, \widecheck\omega_X)$ is
  $B$-semi-invariant of weight $\kappa_P$.
\end{prop}
\begin{proof}
  Since $U$ has trivial Picard group, every $B$-linearized invertible
  sheaf on $U$ is $B$-equivariantly isomorphic to $\Om_U(\chi)$ for
  some $\chi \in \X(B)$. We have maps $\widecheck\omega_U \to
  \Om_U(\chi) \to \Om_U$, where the first map is chosen to be a
  $B$-equivariant isomorphism, and the second map is canonical, but
  not necessarily $B$-equivariant.  Let $f$ be the image of $s$ under
  the composed map. By \cite[Proposition~1.3,~(ii)]{knoppic}, the
  regular function $f$ is $B$-semi-invariant. It follows that $s$ is
  $B$-semi-invariant as well, where the weight has to be corrected by
  the twist $\chi$.

  Next, we determine the weight of $s$. Let $w\in T$,
  $f\in\Gamma(U,\Om_U)$, and $x\in U$. Then we have
  \begin{align*}
    \rleft(\rleft(w\cdot \uf_\alpha\rright)(f)\rright)(x) &=
    \rleft(\rleft(\lambda_w^\#\circ\uf_\alpha\circ\lambda_{w^{-1}}^\#\rright)(f)\rright)(x)\\
    &=\frac{\partial}{\partial t}\bigg|_{e}f(wu(t)w^{-1} \cdot x) \\
    & =\frac{\partial}{\partial t}\bigg|_{e}f(u(\alpha(w)t) \cdot x) \\
    &=\rleft(\rleft(\alpha(w)\uf_\alpha\rright)(f)\rright)(x)\text{.}
  \end{align*}
  Hence $\uf_\alpha$ is $T$-semi-invariant of weight
  $\alpha$. Analogously, one can show that $\lf_i$ is $T$-invariant.
  In particular, the weight of $s$ depends only on the set $S^p$. As
  $G/P^{\minus}$ and $X$ have the same stabilizer of the open
  $B$-orbit, it follows that $s$ is $B$-semi-invariant of weight
  $\kappa_P$.
\end{proof}

\begin{remark}
  \label{rem:kphs}
  For $I \subseteq S$ we denote by $\rho_I$ the half-sum of the
  positive roots in the root system generated by $I$. 
  It is also known that $\rho_I$ is the sum
  of the fundamental dominant weights of the root system generated by $I$ (see
  \cite[Lemma~13.3A]{Humphreys:LieAlgebras}).
  By
  Proposition~\ref{prop:weight_of_special_section}, we have $\kappa_P
  = 2\rho_S - 2\rho_{S^p}$.
\end{remark}

\begin{remark}
  A nonzero $B$-semi-invariant rational section of
  $\widecheck\omega_X$ is uniquely determined by its weight $\chi \in
  \X(B)$ up to a constant factor, and such a rational section exists
  if and only if $\chi \in \kappa_P + \Mm$ (see also Section~\ref{sec:comp-with-moment}).
\end{remark}

\begin{cor}[{\cite[Proposition 4.1]{Brion:cc}}]
  \label{cor:bo}
  We have
  \[
  \Div s=\sum_{i=1}^km_iD_i+\sum_{j=1}^nX_j
  \]
  with $m_i \in \Z_{>0}$.
\end{cor}
\begin{proof}
  Since $s$ is a $B$-semi-invariant section, it follows that its
  divisor is a linear combination of the $B$-invariant divisors of
  $X$, \ie\
  \[
  \Div s=\sum_{i=1}^km_iD_i+\sum_{j=1}^nr_jX_j
  \]
  for integers $m_i$ and $r_j$. Since $s$ vanishes on every
  $B$-invariant divisor, it follows that the $m_i$ and $r_j$ are
  positive. To show that $r_j=1$, we consider the restriction
  $s'\coloneqq s|_{X^0}$. Above, we have seen that the open subset
  $X^0$ is isomorphic to the product variety $R_u(P)\times Z$. In
  particular,
  $\widecheck\omega_{X^0}\cong\pi_1^*\widecheck\omega_{R_u(P)}\otimes\pi_2^*\widecheck\omega_Z$
  where $\pi_i$ denotes the projection onto the $i$-th factor. The
  section $s'$ behaves well under this product decomposition. Indeed,
  set $s_1\coloneqq\bigwedge_{\alpha\in
    R^+\setminus\left<S^p\right>}\uf_\alpha|_{R_u(P)}$ and
  $s_2\coloneqq\lf_1|_Z\wedge\ldots\wedge\lf_r|_Z$. Then under the
  isomorphism above the section $s'$ corresponds to the section
  $s_1\otimes s_2$. Since $s_1$ does not vanish on $R_u(P)$, we obtain
  \[
  \sum_{j=1}^nr_jX_j= (\Div s)|_{X^0} = \Div s' =\Div(s_1\otimes
  s_2)=\pi_2^* \Div s_2.
  \]
  Now, $Z$ is a toric variety with respect to the quotient torus
  $T_1/(T_0\cap T_1)$, and the pullback under $\pi_2$ of its
  torus-invariant divisors are exactly the $G$-invariant divisors
  $X_j$ of $X$. Since $\sigma_2$ is invariant under the action of the
  torus $T_1/(T_0\cap T_1)$, it follows, by toric geometry, that
  $r_j=1$.
\end{proof}

\begin{proof}[{Proof of Theorem~\ref{th:a2}}]
  By Corollary \ref{cor:bo}, it follows that the $B$-semi-invariant
  section $s$ of Theorem~\ref{th:a1} has a zero of order $1$ along any
  $G$-invariant prime divisor in any spherical embedding
  $G/H\hookrightarrow X$.
 
  Now assume that $s'$ is a generator of
  $\Gamma(U,\widecheck{\omega}_{G/H})$ which has a zero of order $1$
  along any $G$-invariant prime divisor in any spherical embedding
  $G/H\hookrightarrow X$. Since the open $B$-orbit $U$ of $G/H$ has
  trivial divisor class group, we have maps $\widecheck\omega_U \to
  \Om_U(\chi) \to \Om_U$, where the first map is chosen to be a
  $B$-equivariant isomorphism, and the second map is canonical, but
  not necessarily $B$-equivariant. Under this isomorphism $s$ and $s'$
  correspond to functions $f,f'\in\Gamma(U,\Om_{G/H})$. Since $s$ and
  $s'$ are generators, the functions $f$ and $f'$ are invertible. In
  particular, $f/f'\in\Gamma(U,\Om_{G/H}^*)$. By \cite[Proposition
  1.3, (ii)]{knoppic}, $f/f'$ is $B$-semi-invariant and hence its
  $B$-weight is contained in $\Mm$.
  
  Take any spherical embedding $G/H\hookrightarrow X$ which contains
  exactly two orbits, namely $G/H$ and a $G$-invariant prime divisor
  $X_1$.  The colored fan corresponding to $X$ is given by the ray
  $\Q_{\ge0}\nu_1$ where $\nu_1\in\Nm$ denotes the $G$-invariant
  valuation induced by $X_1$. Since $s$ and $s'$ have a zero of order
  $1$ along $X_1$, we obtain $\langle f/f',\nu_1\rangle=0$.
  
  It follows that the valuation cone $\Vm$ of $G/H$ is contained in
  the subspace $\{w \in \Nm_\Q : \langle f/f', w \rangle = 0
  \}$. Since $\Vm$ is a full-dimensional cone, this is only possible
  if the $B$-weight of $f/f'$ is $0$, \ie $f$ and $f'$ coincide up to
  a scalar multiple.
\end{proof}

\section{Types of colors}
\label{sec:typcol}

In \cite[4.2]{Brion:cc}, Brion has defined colors of types $\II$,
$\III$, and $\IV$. These types are equivalent to the types $b$, $a$,
and $2a$ due to Luna respectively, which means that
Theorem~\ref{th:ac} is equivalent to \cite[Theorem~4.2]{Brion:cc}. We
will, however, prove Theorem~\ref{th:ac} by reducing it to the
situation of \cite[3.6]{Luna:cc}, \ie to the case of a wonderful
variety (and afterwards it will be easy to see that the types of colors
due to Brion and Luna coincide).

A \emph{wonderful variety} is a spherical variety which is complete,
smooth, simple, and toroidal. We explain (from~\cite[6.1]{Luna:typea})
how to associate to the spherical variety $X$ a wonderful variety $Y$.
We identify the $G$-equivariant automorphism group of $G/H$ with
$N_G(H)/H$. Then $N_G(H)$ acts on $\Dm$, and we define $\overline{H}
\subseteq N_G(H)$ to be the kernel of this action. It contains $H$ and
is called the \emph{spherical closure} of $H$. There exists a (unique)
spherical embedding $G / \overline{H} \hookrightarrow Y$ such that $Y$
is a wonderful variety. We denote its set of colors by $\overline{\Dm}
= \{ \overline{D}_1, \dots, \overline{D}_k\}$, which is in bijection
with the set of colors $\Dm = \{D_1, \ldots, D_k\}$ of $G/H$ via
$\pi\colon G/H \to G/\overline{H}$.  We denote the stabilizer of the
open $B$-orbit in $Y$ by $\overline{P}$ and the coefficients of the
colors in the expression for the anticanonical divisor from
Theorem~\ref{th:a1} by $\overline{m}_1, \ldots, \overline{m}_k$.

\begin{theorem}[{\cite[3.6]{Luna:cc}}]
  We have
  \begin{align*}
    \overline{m}_i &= r(\overline{D}_i) \coloneqq \tfrac{1}{2} \langle \alpha^\vee,
    \kappa_{\overline{P}} \rangle = 1
    &&\text{ for $\overline{D}_i$ of type $a$ or $2a$,}\\
    \overline{m}_i &= r(\overline{D}_i) \coloneqq \langle \alpha^\vee, \kappa_{\overline{P}}
    \rangle \ge 2 && \text{ for $\overline{D}_i$ of type $b$.}
  \end{align*}
\end{theorem}
\begin{proof}
  By \cite[Proposition~3.6(2)]{Luna:cc}, an anticanonical divisor of
  $Y$ is given by $\sum_{i=1}^k r(\overline{D}_i) \overline{D}_i +
  \sum_{j=1}^{l} Y_j$ where $Y_1, \ldots, Y_{l}$ are the $G$-invariant
  prime divisors in $Y$. The result now follows from the fact that
  $\overline{\Dm}$ is a basis of $\Pic(Y) = \Cl(Y)$.
\end{proof}

\begin{proof}[Proof of Theorem~\ref{th:ac}]
  Let $s \in \Gamma(G/H, \check{\omega}_{G/H})$ and $\overline{s} \in
  \Gamma(G/\overline{H}, \check{\omega}_{\smash{G/\overline{H}}})$ be
  $B$-semi-invariant sections of weights $\kappa_P$ and
  $\kappa_{\overline{P}}$ respectively.  As $\overline{P} = P$, we
  have $\pi^*\overline{s} = s$ (up to a constant), and hence, for the
  pullback of Cartier divisors, $\pi^*\Div \overline{s} = \Div s$.  On
  the other hand, we have the pullback of Cartier divisors
  $\pi^*\overline{D}_i = D_i$ (apply \cite[Section~2.2,
  Theorem~2.2]{foschi}, see also \cite[Lemma~30.24]{ti}), hence $m_i =
  \overline{m}_i$.
\end{proof}

\begin{remark}
  \label{rem:gssc}
  Replacing $G$ with a finite cover, we may assume
  $G = G^{ss} \times C$ where $G^{ss}$ is semisimple
  simply-connected and $C$ is a torus. Then, by
  \cite[Proposition 2.4 and Remark after it]{KKLV},
  every invertible sheaf on the normal variety $X$ can
  be $G$-linearized. 
  There exist (unique) $G$-linearizations of the invertible
  sheaves $\Om_X(D_i)$ and $\Om_X(X_j)$ such that their canonical
  sections are $C$-invariant. With these linearizations,
  Brion's description 
  of the anticanonical sheaf
  \begin{align*}
    \label{eq:1}
    \check{\omega}_X = \Om_X( D_1 )^{\otimes m_1} \otimes \ldots \otimes
    \Om_X( D_k )^{\otimes m_k} \otimes \Om_X( X_1 ) \otimes \ldots \otimes
    \Om_X( X_n )
  \end{align*}
  of an arbitrary spherical variety $X$ is not only valid inside $\Pic
  ( X_\mathrm{reg} )$, but even inside the group of isomorphism
  classes of $G$-linearized invertible sheaves $\Pic^G( X_\mathrm{reg}
  )$.
\end{remark}

\begin{prop}
  The types of colors $\II, \III$, and $\IV$ due to Brion coincide
  with the types of colors $b$, $a$, and $2a$ due to Luna
  respectively.
\end{prop}
\begin{proof}
  A color $D_i\in \Dm$ is of type $\II$ if and only if it is of type
  $b$ as both situations are characterized by $m_i \ge 2$. Now let $D
  \in \Dm$ be of type $\III$ (resp.~of type $\IV$), and let $C$ be a
  basic curve representing $D$ in the sense of
  \cite[4.2]{Brion:cc}. The kernel of the $B$-action on $C$ is the
  radical of a (unique) minimal parabolic subgroup $P_\alpha$ (see
  \cite[1.1]{Brion:cc}), and we have $D \in \Dm(\alpha)$ (see
  \cite[Proposition~3.6]{sphmori}). According to
  \cite[Proposition~1.2]{Brion:cc}, the spherical variety $G\cdot C$
  contains a color of type $a$ (resp.~of type $2a$) which is moved by
  $P_\alpha$. By \cite[Theorem~1.1]{gh14}, this is only possible if
  $\alpha \in \Sigma$ (resp.~if $2\alpha \in \Sigma$), \ie $D$ is of
  type $a$ (resp.~of type $2a$).
\end{proof}

Finally, it will be helpful to explain one further approach due to Knop
to characterize the types of colors, which is easier to apply in
certain situations (see the examples in Section~\ref{section:aex}). We
fix a point $x_0 \in U$ in the open $B$-orbit. Let $\alpha \in S$ such
that $\Dm(\alpha) \ne \emptyset$.  We have $P_\alpha/B \cong \Pb^1$
and the natural action of $P_\alpha$ on $\Pb^1$ yields a morphism
$\phi_\alpha \colon P_\alpha \to \PGL_2$. Let $H_\alpha$ be the
stabilizer of $x_0$ inside $P_\alpha$. Then $\phi_\alpha(H_\alpha)$ is
a proper spherical subgroup of $\PGL_2$ (see \cite[Lemma
3.1]{Knop:Borbits}), \ie $\phi_\alpha(H_\alpha)$ is either a maximal
torus, the normalizer of a maximal torus, or contains a maximal
unipotent subgroup (see \cite[Lemma 3.2]{Knop:Borbits}).

\begin{theorem}[{\cite[Section~2]{Knop:loc}}]
  \label{th:atypes}
  A color $D \in \Dm(\alpha)$ is
  \begin{align*}
    &\text{of type $a$ } && \text{ if $\Phi_\alpha(H_\alpha)$ is a
      maximal torus,}\\
    &\text{of type $2a$ } && \text{ if $\Phi_\alpha ( H_\alpha )$ is
      the normalizer of a maximal torus,}\\
    &\text{of type $b$ } && \text{ if $\Phi_\alpha(H_\alpha)$ contains
      a maximal unipotent subgroup.}
  \end{align*}
\end{theorem}

\section{Examples illustrating Theorem~\ref{th:ac}}
\label{section:aex}
In this section, we compute the coefficients $m_i$ for several
well-known spherical homogeneous spaces.  In \ref{ex:ac1}, \ref{ex:ac2},
and \ref{ex:ac3} we give
examples for every type of color. Example~\ref{ex:ac4} illustrates $S^p \ne
\emptyset$ and $m_i > 2$. Observe that in \ref{ex:ac1}, \ref{ex:ac3},
and \ref{ex:ac4} the
computation is simplified by applying Theorem~\ref{th:atypes}.

\begin{example}
  \label{ex:ac1}
  Consider $G \coloneqq \SL_n \times \SL_n$ and $H \coloneqq \SL_n
  \subseteq G$ diagonally embedded. Then $G/H$ is isomorphic to
  $\SL_n$ where $G$ acts with the first factor from the left and with
  the second factor from the right after inverting.  We denote by
  $\operatorname{D}_n \subseteq \SL_n$ the subgroup of diagonal
  matrices, by $\operatorname{T}_n \subseteq \SL_n$ the subgroup of
  upper triangular matrices, and by $\operatorname{T}_n^- \subseteq
  \SL_n$ the subgroup of lower triangular matrices.  We define $B
  \coloneqq \operatorname{T}_n^- \times \operatorname{T}_n$ and $T
  \coloneqq \operatorname{D}_n \times \operatorname{D}_n$, and
  obtain the set of simple roots $S = \{\alpha_1, \ldots,
  \alpha_{n-1}, \beta_1, \ldots, \beta_{n-1}\}$.

  The homogeneous space $G/H$ is spherical, and there are $n-1$ colors
  of $G/H$, i.e. $\Dm = \{D_1, \ldots, D_{n-1}\}$, given by $D_i =
  \V(f_i)$ where $f_i \in \C[\SL_n]$ is the upper-left principal minor
  of size $i \times i$.  It is not difficult to see that
  $\Dm(\alpha_i) = \Dm(\beta_i) = \{D_i\}$.  Furthermore, all colors
  are of type $b$ since for every $\alpha_i$ the image of $H \cap
  P_{\alpha_i}$ under $\Phi_{\alpha_i} \colon P_{\alpha_i} \to \PGL_2$
  is a Borel subgroup of $\PGL_2$ and therefore contains a maximal
  unipotent subgroup.

  The stabilizer of the open $B$-orbit in $G/H$ is $B$
  itself. Therefore we have $S^p = \emptyset$ and $\kappa_P =
  2\rho_S$. Since $\rho_S$ is equal to the sum of the fundamental
  dominant weights, we obtain $m_i = \langle \alpha_i^\vee,
  2\rho_S\rangle = 2$ for $1\le i \le n-1$.
\end{example}

\begin{example}
  \label{ex:ac2}
  Consider $G\coloneqq\SL_2\times\SL_2\times\SL_2$ and
  $H\coloneqq\SL_2\subseteq G$ diagonally embedded. We denote by $B
  \subseteq G$ the Borel subgroup of lower triangular matrices and by
  $T \subseteq B$ the maximal torus of diagonal matrices. We obtain
  the set of simple roots $\{\alpha,\beta,\gamma\}$ corresponding to
  the three factors in $G$.
 
  The homogeneous space $G/H$ is spherical, and there are $3$ colors
  of $G/H$, \ie $\Dm=\{D_{12},D_{13},D_{23}\}$. To be more precise,
  let $A_{ij}$ for $1\le i<j\le 3$ be the $2\times 2$ matrix whose
  rows are given by the first rows of the $i$-th and $j$-th factor of
  $G$. Then $\det A_{ij}$ is an equation for $D_{ij}$. It is not
  difficult to see that $\Dm(\alpha)=\{D_{12}$, $D_{13}\}$,
  $\Dm(\beta)=\{D_{12},D_{23}\}$, and
  $\Dm(\gamma)=\{D_{13},D_{23}\}$. In particular, it follows that all
  colors are of type $a$, and therefore $m_{ij}=1$ for $1 \le i < j
  \le 3$.

  By the definition of the types of colors due to Luna (see
  Section~\ref{sec:intr}), it follows that $\alpha, \beta, \gamma$ are
  contained in $\Mm$, and the valuation cone is given by
  \begin{align*}
    \Vm=\{v\in\Nm_\Q : \langle v, \alpha \rangle \le 0, \langle v,
    \beta \rangle \le 0 , \langle v, \gamma \rangle \le 0\}\text{.}
  \end{align*}
\end{example}

\begin{example}
  \label{ex:ac3}
  Consider $G \coloneqq \SL_n$ for $n \ge 3$ and $H\coloneqq \SO_n$,
  the subgroup of orthogonal matrices.  Let $G$ act on the space
  $\Sym(n)$ of symmetric $n\times n$ matrices via $A \cdot M =
  AMA^T$. Then the stabilizer of the identity matrix is $H$.  We
  denote by $B \subseteq G$ the Borel subgroup of lower triangular
  matrices and by $T \subseteq B$ the maximal torus of diagonal
  matrices. We obtain the set of simple roots $\{\alpha_1, \ldots,
  \alpha_{n-1}\}$.

  The homogeneous space $G/H$ is spherical, and there are $n-1$ colors
  of $G/H$, i.e. $\Dm = \{D_1, \ldots, D_{n-1}\}$, given by $D_i =
  \V(f_i)$ where $f_i \in \C[\SL_n]$ is again the upper-left principal
  minor of size $i \times i$.  It is not difficult to see that
  $\Dm(\alpha_i) = \{D_i\}$.  Furthermore, all colors are of type $2a$
  since for every $\alpha_i$ the image of $H \cap P_{\alpha_i}$ under
  $\Phi_{\alpha_i} \colon P_{\alpha_i} \to \PGL_2$ is the normalizer
  of a maximal torus.  It follows that $m_i = 1$ for $1 \le i \le
  n-1$.

  As in the previous example, it follows that $2\alpha_i$ is contained
  in $\Mm$ for $1 \le i\le n-1$, and the valuation cone is given by
  \begin{align*}
    \Vm=\{v\in\Nm_\Q : \langle v, 2\alpha_i \rangle\le0 \text{ for
      every } 1 \le i \le n-1\}\text{.}
  \end{align*}
\end{example}

\begin{example}
  \label{ex:ac4}
  Consider $G \coloneqq \SL_n$ for $n \ge 3$ and $H \coloneqq
  \SL_{n-1}$ embedded as the block diagonal matrices with entries on
  the lower-right of $\SL_n$. Let $B \subseteq G$ be the Borel
  subgroup of upper triangular matrices and $T \subseteq G$ the
  subgroup of diagonal matrices. We obtain the set of simple roots $S
  = \{\alpha_1, \ldots, \alpha_{n-1}\}$.

  Let $G$ act on $\C^n \times \C^n$ by acting naturally on the first
  factor and with the contragredient action on the second
  factor. Denoting the coordinates of the first factor by $X_1,
  \ldots, X_n$ and the coordinates of the second factor by $Y_1,
  \ldots Y_n$, we obtain
  \begin{align*}
    G/H \cong \V(X_1Y_1 + \ldots + X_nY_n - 1) \subseteq \C^n \times
    \C^n\text{.}
  \end{align*}
  There are two colors $D_1 \coloneqq \V(X_n)$ and $D_2 \coloneqq
  \V(Y_1)$.  It is not difficult to see that $\Dm(\alpha_{n-1}) =
  \{D_1\}$, $\Dm(\alpha_{1}) = \{D_2\}$, and $S^p = \{\alpha_2,
  \ldots, \alpha_{n-2}\}$.  We see that both colors are of type $b$ as
  in the previous example, but we have to suitably conjugate $H$ first
  since the basepoint does not lie in the open $B$-orbit. We obtain
  \begin{align*}
    m_1 = \langle \alpha_{n-1}^\vee, 2\rho_S - 2\rho_{S^p}\rangle =
    \rleft\langle \alpha_{n-1}^\vee,
    \sum_{i=1}^{n-1}\sum_{j=1}^i\alpha_j+\sum_{i=1}^{n-2}\sum_{j=1}^i\alpha_{n-j}\rright\rangle
    =n-1\text{,}
  \end{align*}
  and similarly $m_2 = n-1$.
\end{example}

\section{$\Q$-Gorenstein spherical Fano varieties}
\label{sec:spherical-q-fano}

We continue to use the notation from the introduction.
In particular, $G/H$ is a spherical homogeneous space,
$\Dm = \{D_1, \ldots, D_k \}$ is the set of colors, and $m_i \in \Z_{>0}$ are
the coefficients of the colors in the expression for the anticanonical
divisor from Theorem~\ref{th:a1}.

\begin{definition}
  \label{def:qrefl}
  A polytope $Q \subseteq \Nm_\Q$ is called
  \emph{$\Q$-$G/H$-reflexive} if the following conditions are
  satisfied:
  \begin{enumerate}
  \item $\rho( D_i ) / m_i \in Q$ for every $i = 1, \ldots, k$.
  \item $0 \in \topint(Q)$.
  \item Every vertex of $Q$ is contained in $\{ \rho( D_i ) / m_i : i
    = 1, \ldots, k \}$ or a primitive element in $\Nm \cap \Vm$.
  \end{enumerate}
\end{definition}

\begin{prop}
  \label{prop:A_X_Q_G_mod_H_relfexive}
  Let $G/H \hookrightarrow X$ be a $\Q$-Gorenstein spherical Fano
  embedding. Then the polytope $Q_X \subseteq \Nm_\Q$ is
  $\Q$-$G/H$-reflexive.
\end{prop}
\begin{proof}
  Let $X_1,\ldots, X_n$ be the $G$-invariant prime divisors in $X$.
  It follows from the completeness of $X$ that
  \begin{align*}
    \cone( \rho( D_1 ), \ldots, \rho( D_k ), \nu_{X_1}, \ldots,
    \nu_{X_n} ) = \Nm_\Q\text{,}
  \end{align*}
  and therefore $0 \in \topint(Q_X)$.
\end{proof}

\begin{prop}
  \label{prop:X_associated_A_Q_Fano}
  Let $Q\subseteq\Nm_\Q$ be a $\Q$-$G/H$-reflexive polytope. Then
  $\Ff_{\Supp}(Q)$ is a complete supported colored fan such that the
  associated spherical embedding $G/H\hookrightarrow
  X_{\Ff_{\Supp}(Q)}$ is $\Q$-Gorenstein Fano.
\end{prop}
\begin{proof}
  As $Q$ is full-dimensional, $\Ff( Q )$ is complete (see
  Definition~\ref{def:colored_face_fan}), and therefore, by
  Remark~\ref{rem:s_maps_complete_to_complete}, $\Ff_{\Supp}( Q )$ is
  complete as well. Let $X \coloneqq X_{\Ff_{\Supp}(Q)}$. Every
  maximal cone $( \Cm, \Fm ) \in \Ff_{\Supp} ( Q )$ is given by $\Cm =
  \cone ( \widehat{v_\Cm} )$ and $\Fm = \rho^{-1} ( \widehat{v_\Cm} )
  $ for a vertex $v_\Cm \in Q^*$. We define a piecewise linear
  function $\psi \colon \Vm \to \Q$ by $\psi|_{\Cm\cap \Vm} \coloneqq
  - \langle \cdot, v_\Cm \rangle$.  It is straightforward to check
  that $\psi$ is the piecewise linear function corresponding to the
  $\Q$-Cartier divisor $-K_X$.  As $Q$ is a polytope, the piecewise
  linear function $\psi$ is strictly convex. Together with property
  (1) of Definition~\ref{def:qrefl}, it follows from
  Proposition~\ref{prop:ample} that $-K_X$ is ample, and hence
  $X=X_{\Ff_{\Supp}(Q)}$ is Fano.
\end{proof}

\begin{prop}
  The assignments $X \mapsto Q_X$ and $Q \mapsto X_{\Ff_{\Supp}( Q )}$
  define a bijection between isomorphism classes of
  $\Q$-Gorenstein spherical Fano embeddings of $G/H$ and $\Q$-$G/H$-reflexive
  polytopes.
\end{prop}
\begin{proof}
  The well-definedness of the two maps follows from
  Proposition~\ref{prop:A_X_Q_G_mod_H_relfexive} and
  Proposition~\ref{prop:X_associated_A_Q_Fano}. It remains to show
  that the maps are inverse to each other.
  
  Let $Q \subseteq \Nm_\Q$ be a $\Q$-$G/H$-reflexive polytope and $G/H
  \hookrightarrow X$ the spherical embedding corresponding to the
  colored fan $\Ff_{\Supp} ( Q )$. Let $u \in Q$ be a vertex. By
  property (3) of Definition~\ref{def:qrefl}, we have $u \in Q_X$,
  hence $Q \subseteq Q_X$.  Now, let $u \in Q_X$ be a vertex. If $u =
  \tfrac{\rho(D_i)}{m_i}$ for some $i \in \{1, \ldots, k\}$, then $u
  \in Q$ because of property (1) of
  Definition~\ref{def:qrefl}. Otherwise, the vertex $u$ is the
  primitive generator of a ray in $\Ff_{\Supp}(Q)$, in which case the
  definition of the face fan also implies $u \in Q$. Hence we have
  $Q_X \subseteq Q$.

  Let $G/H \hookrightarrow X$ be a $\Q$-Gorenstein spherical Fano
  embedding with associated supported colored fan $\Ff$ and $\psi
  \colon \Vm \to \Q$ the strictly convex piecewise linear function
  associated to the anticanonical divisor $-K_X$.  It suffices to show
  that the maximal cones of $\Ff_{\Supp} ( Q_X )$ and $\Ff$
  coincide. Let $(\Cm, \Fm)$ be a maximal cone in $\Ff$. Then
  $\psi|_{\Cm\cap \Vm} = -\langle \cdot, v_\Cm \rangle$ for some
  $v_\Cm \in \Mm_\Q$.  We obtain that $F \coloneqq \{ u \in Q_X :
  \langle u, v \rangle = -1 \}$ is a facet of $Q_X$ such that $\cone(
  F ) = \Cm$ and $\Fm = \rho^{-1}( F )$. In particular, it follows
  that all maximal cones of $\Ff$ are in $\Ff_{\Supp}( Q_X )$.
\end{proof}

\section{Gorenstein spherical Fano varieties}
\label{sec:spher-gorenst-fano}

Let $Q \subseteq \Nm_\Q$ be a $\Q$-$G/H$-reflexive polytope. In this
section, we investigate when the associated $\Q$-Gorenstein
spherical Fano embedding $G/H \hookrightarrow X$ is Gorenstein, \ie the
anticanonical divisor is Cartier. Recall that the vertices of $Q^*$
correspond to the colored cones of maximal dimension in $\Ff(Q)$.

\begin{definition}
  A vertex $v \in Q^*$ is called \emph{supported} if the corresponding
  colored cone in $\Ff(Q)$ is supported. In this case, the
  corresponding facet $\widehat{v} \preceq Q$ is called
  \emph{supported} as well.  The set of supported vertices of $Q^*$ is
  denoted by $V_{\Supp}( Q^* )$.
\end{definition}

\begin{lemma}
  \label{le:msupp}
  Let $\Cm \subseteq \Nm_\Q$ be a cone of maximal dimension. Then the
  following statements are equivalent:
  \begin{enumerate}
  \item $\relint(\Cm) \cap \Vm = \emptyset$.
  \item There exists $v \in \Mm_\Q$ such that $\langle \cdot,
    v\rangle|_{\Cm} \ge 0$ and $\langle \cdot, v \rangle|_{\Vm} \le
    0$.
  \item $\cone( \Sigma ) \cap \Cm^\vee \neq \{ 0 \}$.
  \end{enumerate}
\end{lemma}
\begin{proof}
  (1) $\Rightarrow$ (2) follows from the Hahn-Banach separation
  theorem since $\relint(\Cm)$ is open in $\Nm_\Q$, and (2)
  $\Rightarrow$ (3) is obvious. In order to show (3) $\Rightarrow$
  (1), let $0 \ne v \in \cone( \Sigma ) \cap \Cm^\vee$ and $u \in
  \relint(\Cm)$. As $\Cm$ is full-dimensional, \ie $\Cm^\perp = \{ 0
  \}$, we obtain $\langle u, v \rangle > 0$. It follows that $u \notin
  \Vm$.
\end{proof}

\begin{prop}
  \label{prop:suppvert}
  A vertex $v \in Q^*$ is supported if and only if $Q^* \cap (v +
  \cone( \Sigma )) = \{v\}$.
\end{prop}
\begin{proof}
  Let $v \in Q^*$ be a vertex and $\Cm \coloneqq \cone ( \widehat{v} )
  \subseteq \Nm_\Q$ the corresponding cone in $\Ff(Q)$.  We denote by
  \begin{align*}
    \Ts_vQ^* \coloneqq \{ v + \lambda \cdot ( v' - v) : v' \in Q^*,
    \lambda \ge 0\}
  \end{align*}
  the tangent cone of $Q^*$ in $v$, which is an affine cone with apex
  $v$.  By \cite[Corollary~5.2.5]{FundamentalsConvexAnalysis}, we have
  $\Ts_vQ^* - v = \Cm^\vee$ since $\Cm$ is the normal cone of $Q^*$
  along the vertex $v$.  The statement now follows from
  Lemma~\ref{le:msupp} since $\Tm_v Q^* \cap (v + \cone( \Sigma )) =
  \{v\}$ if and only if $Q^* \cap (v + \cone( \Sigma )) = \{v\}$.  The
  last statement follows since any $v' \in \Tm_v Q^* \cap (v + \cone(
  \Sigma ))$ may be rescaled to be arbitrarily near to $v$.
\end{proof}

\begin{definition}
  The $\Q$-$G/H$-reflexive polytope $Q$ is called
  \emph{$G/H$-reflexive} if every supported vertex of $Q^*$ lies in
  the lattice $\Mm$.
\end{definition}

We will show in Proposition~\ref{prop:refag} that this definition is
in agreement with Definition~\ref{def:qghrefl}.

\begin{theorem}
  Recall that $X$ is a $\Q$-Gorenstein spherical Fano variety by
  assumption.
  The variety $X$ is Gorenstein if and only if the polytope $Q$ is
  $G/H$-reflexive. In particular, there is a bijection between
  isomorphism classes of Gorenstein spherical Fano embeddings of $G/H$
  and $G/H$-reflexive polytopes.
\end{theorem}
\begin{proof}
  Let $\psi \colon \Vm \to \Q$ be the piecewise linear function
  corresponding to the $\Q$-Cartier divisor $-K_X$.  If $\Cm_v
  \subseteq \Nm_\Q$ is the cone corresponding to the vertex $v \in
  Q^*$, we have $\psi|_{\Cm_v \cap \Vm} = -\langle \cdot, v\rangle$.
  Then $-K_X$ is Cartier if and only if $v \in \Mm$ for every $v$
  where $\Cm_v$ is in $\Ff_{\Supp}(Q)$, \ie if and only if the
  supported vertices of $Q^*$ lie in the lattice $\Mm$.
\end{proof}

\begin{prop}
  \label{prop:refag}
  A polytope $Q \subseteq \Nm_\Q$ is $G/H$-reflexive if and only if
  the following conditions are satisfied:
  \begin{enumerate}
  \item $\rho( D_i ) / m_i \in Q$ for every $i = 1, \ldots, k$.
  \item $0 \in \topint(Q)$.
  \item Every vertex of $Q$ is contained in $\{ \rho( D_i ) / m_i : i
    = 1, \ldots, k \}$ or $\Nm \cap \Vm$.
  \item Every supported vertex of $Q^*$ lies in the lattice $\Mm$.
  \end{enumerate}
\end{prop}
\begin{proof}
  Let $Q \subseteq \Nm_\Q$ be a full-dimensional polytope satisfying
  the above conditions and $u \in Q$ a vertex not contained in
  $\{\rho(D_i)/m_i : i=1, \ldots, k \}$. Then we have $u \in \Vm$,
  which implies that there exists a supported full-dimensional cone
  $\Cm$ in $\Ff(Q)$ having $\Q_{\ge 0}u$ as extremal ray. Let $v \in
  Q^*$ be the supported vertex corresponding to $\Cm$.  Then $\langle
  u, v \rangle = -1$, and, as $u \in \Nm$ and $v \in \Mm$, we obtain
  that $u$ is primitive.
\end{proof}

\section{Global sections of the anticanonical sheaf}
\label{sec:comp-with-moment}

In this section, we recall from \cite[3.3]{l-brion-pic}
the $G$-module structure of the space of
global sections of a $B$-invariant Cartier divisor on
 a spherical variety $X$. For simplicity, we assume that $X$
is complete. We then
investigate the special case of the anticanonical sheaf 
when $X$ is Gorenstein Fano. We use \cite[Section 17.4]{ti}
as a general reference.

Let $G/H \hookrightarrow X$ be an arbitrary spherical
embedding with associated supported colored fan $\Ff$.
We denote by $D_1, \ldots, D_k$ the colors and by $X_1, \ldots, X_n$
the $G$-invariant prime divisors in $X$.
Consider a $B$-invariant Cartier divisor
\begin{align*}
\delta \coloneqq \sum_{i=1}^k a_i D_i + \sum_{j=1}^n b_j X_j
\end{align*}
on $X$. For every maximal supported colored cone $(\Cm, \Fm) \in \Ff$
we write $v_\Cm \in \Mm$ as in the last part of
Section~\ref{sec:notat-gener}. As in Remark~\ref{rem:gssc}, we
may assume $G = G^{ss} \times C$ where $G^{ss}$ is semisimple
simply-connected and $C$ is a torus, and then the invertible
sheaf $\Om_X(\delta)$ can
be $G$-linearized.  Let $s_\delta$ be a rational section of $\Om_X(
\delta )$ satisfying $\Div s_\delta = \delta$. As $\delta$ is
$B$-invariant, the section $s_\delta$ is $B$-semi-invariant of some
weight $\kappa_\delta \in \X(B)$ (not necessarily contained in $\Mm$).
We write $\Ff_{\max}$ for the set of maximal cones of $\Ff$, and we
set $\Dm_X \coloneqq \bigcup_{(\Cm,\Fm) \in \Ff_{\max}}\Fm$.  If we
define
\begin{align*}
  P_\delta \coloneqq \rleft\{ u \in \bigcap_{(\Cm,\Fm) \in \Ff_{\max}} \rleft(
  -v_\Cm + \Cm^\vee \rright) : \langle \rho(D), u \rangle \ge -m_D
  \text{ for every } D \in \Dm \setminus \Dm_X\rright\}\text{,}
\end{align*}
then $( \kappa_\delta + P_\delta ) \cap \X(B)$ is contained in the set
of dominant weights, and we have
\begin{align*}
  \Gamma( X, \Om_X(\delta) ) \cong \bigoplus_{\chi \in ( \kappa_\delta
    + P_\delta ) \cap \X(B)} V_\chi
\end{align*}
where $V_\chi$ denotes the irreducible $G$-module of highest weight
$\chi$.

\begin{remark}
  Let $G/H \hookrightarrow X$ be a Gorenstein spherical Fano embedding
  with associated $G/H$-reflexive polytope $Q$, and let $\delta$ be
  the anticanonical divisor of Theorem \ref{th:a1}
  equipped with the canonical $G$-linearization.
  It is straightforward to check that
  we have $\kappa_\delta = \kappa_P$ and $P_\delta = Q^*$.
\end{remark}

\section{Examples illustrating Theorem~\ref{thm:bijection-Fano-poly}}
\label{sec:ex-poly}

\begin{example}
  \label{ex:SL_2_mod_N}
  Let $G \coloneqq \SL_2 \times \C^*$ and consider $H \coloneqq N
  \times \{1\}$ where $N \subseteq \SL_2$ is the normalizer of a
  maximal torus. Fix some maximal torus contained in some Borel subgroup,
  and denote by $\alpha$
  the unique simple root of $\SL_2$. Denote by $\varepsilon$ a primitive
  character of $\C^*$.  Then there is exactly one spherical root
  $\gamma \coloneqq 2\alpha$ and $(\gamma, \varepsilon)$ is a basis of
  the lattice $\Mm$.  We denote by $(\gamma^*, \varepsilon^*)$ the
  corresponding dual basis of the lattice $\Nm$.  There is exactly one
  color $D_1$ (of type $2a$) with $\rho(D_1) = 2\gamma^*$.  Then $Q
  \coloneqq \conv(2\gamma^*, \varepsilon^*, -\gamma^*, -\varepsilon^*)
  \subseteq \Nm_\Q$ is a $G/H$-reflexive polytope, and its dual
  polytope is $Q^* = \conv(\gamma-\varepsilon, \gamma+\varepsilon,
  -\tfrac{1}{2}\gamma+\varepsilon, -\tfrac{1}{2}\gamma-\varepsilon)$.
  The polytopes $Q$ and $Q^*$ are illustrated in
  Figure~\ref{fig:sl2n}.  The valuation cone is shown in grey, and the
  dashed arrow is the image of the color under $\rho$ in $\Nm$.  The
  dotted arrows are translates of the spherical root $\gamma \in \Mm$
  showing that exactly the circled vertices of $Q^*$ are supported
  (see~Proposition~\ref{prop:suppvert}).
  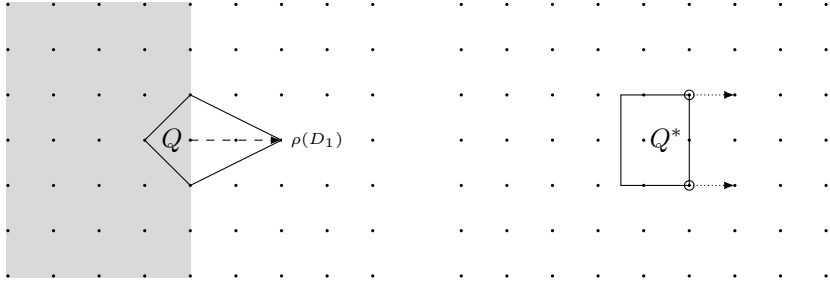
\begin{figure}[ht!]
    \begin{tikzpicture}[scale=0.6]
      \clip (-4.04, -3.04) -- (4.04, -3.04) -- (4.04, 3.04) -- (-4.04, 3.04) -- cycle;
      \fill[color=gray!30] (0, -5) -- (0, 5) -- (-5, 5) -- (-5, -5) -- cycle;
      \draw[draw=black] ( 2, 0 ) -- ( 0, 1 ) -- ( -1, 0 ) -- ( 0, -1 ) -- cycle;
      \foreach \x in {-4,...,4} \foreach \y in {-4,...,4} \fill (\x, \y) circle (1pt);
      \node at (-0.4,0) {$Q$};
      \draw[dashed,-latex] (0, 0) -- (2, 0);
      \node[right,fill=white] at (2,0) {\tiny$\rho(D_1)$};
    \end{tikzpicture}\hspace*{1cm}
    \begin{tikzpicture}[scale=0.6]
      \clip (-4.04, -3.04) -- (4.04, -3.04) -- (4.04, 3.04) -- (-4.04, 3.04) -- cycle;
      \draw[draw=black] ( -0.5, -1 ) -- ( 1, -1 ) -- ( 1, 1 ) -- ( -0.5, 1 ) -- cycle;
      \foreach \x in {-4,...,4} \foreach \y in {-4,...,4} \fill (\x, \y) circle (1pt);
      \draw (1, 1) circle (3pt);
      \draw (1, -1) circle (3pt);
      \draw[densely dotted,-latex] (1, 1) -- (2, 1);
      \draw[densely dotted,-latex] (1, -1) -- (2, -1);
      \node at (0.5,0) {$Q^*$};
    \end{tikzpicture}
    \caption{Illustration to Example~\ref{ex:SL_2_mod_N}.}
    \label{fig:sl2n}
  \end{figure}
\end{example}

\begin{example}
  \label{ex:r2}
  Let $G \coloneqq \Spin_5 \times \Spin_5$. 
  Fix some maximal torus contained in some Borel subgroup,
  and denote by $\alpha_1, \alpha_2$ (resp.~$\alpha'_1, \alpha'_2$)
  the simple roots of the first (resp.~the second) simple factor $\Spin_5$
  where $\alpha_2$ (resp.~$\alpha_2'$) is the shorter root.
  According to the third entry in
  \cite[Table~B]{r2}, there exists a spherical homogeneous space $G/H$
  with spherical roots $\gamma_1 \coloneqq \alpha_2+\alpha'_2$ and
  $\gamma_2 \coloneqq \alpha_1+\alpha'_1$, such that $(\gamma_1,
  \gamma_2)$ is a basis of the lattice $\Mm$. We denote by
  $(\gamma^*_1, \gamma^*_2)$ the corresponding dual basis of the
  lattice $\Nm$.  We have $S^p = \emptyset$, and there are exactly two
  colors $D_1$, $D_2$ (both of type $b$) with $\rho(D_1) = -\gamma^*_1
  + 2\gamma^*_2 $ and $\rho(D_2) = 2\gamma^*_1 -2\gamma^*_2$.  As $S^p
  = \emptyset$, we have $\kappa_P = 2\rho_S$ (see
  Remark~\ref{rem:kphs}), so that the coefficients in the expression
  for the canonical divisor are $m_1 = m_2 = 2$. Then $Q \coloneqq
  \conv(\gamma^*_1-\gamma^*_2, -\tfrac{1}{2}\gamma^*_1+\gamma^*_2,
  -\gamma^*_1, -\gamma^*_2)$ is a $G/H$-reflexive polytope, and its
  dual polytope is $Q^* = \conv(\gamma_1-\tfrac{1}{2}\gamma_2,
  \gamma_1+\gamma_2, \gamma_2, -4\gamma_1-3\gamma_2)$.  The polytopes
  $Q$ and $Q^*$ are illustrated in Figure~\ref{fig:spin5}.  The
  valuation cone is shown in grey, and the dashed arrows are the
  images of the colors under $\rho$ in $\Nm$.  The dotted arrows are
  translates of the spherical roots $\gamma_1, \gamma_2 \in \Mm$
  showing that exactly the circled vertex of $Q^*$ is supported
  (see~Proposition~\ref{prop:suppvert}).

  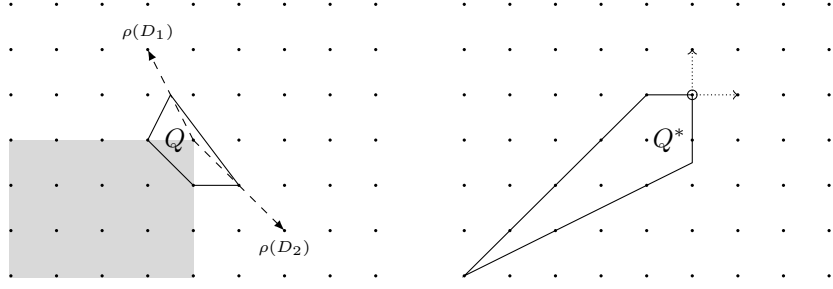
\begin{figure}[ht!]
    \begin{tikzpicture}[scale=0.6]
      \clip (-4.04, -3.04) -- (4.04, -3.04) -- (4.04, 3.04) -- (-4.04, 3.04) -- cycle;
      \fill[color=gray!30] (0, -5) -- (0, 0) -- (-5, 0) -- (-5, -5) -- cycle;
      \draw[draw=black] ( -0.5, 1 ) -- ( 1, -1 ) -- ( 0, -1 ) -- ( -1, 0 ) -- cycle;
      \foreach \x in {-4,...,4} \foreach \y in {-4,...,4} \fill (\x, \y) circle (1pt);
      \draw[dashed,-latex] (0, 0) -- (-1, 2);
      \draw[dashed,-latex] (0, 0) -- (2, -2);
      \node at (-0.4,0) {$Q$};
      \node[above] at (-1,2) {\tiny$\rho(D_1)$};
      \node[below] at (2,-2) {\tiny$\rho(D_2)$};
    \end{tikzpicture}\hspace*{1cm}
    \begin{tikzpicture}[scale=0.6]
      \clip (-4.04, -3.04) -- (4.04, -3.04) -- (4.04, 3.04) -- (-4.04, 3.04) -- cycle;
      \draw[draw=black] ( 1, -0.5 ) -- ( 1, 1 ) -- ( 0, 1 ) -- ( -4, -3 ) -- cycle;
      \foreach \x in {-4,...,4} \foreach \y in {-4,...,4} \fill (\x, \y) circle (1pt);
      \draw (1, 1) circle (3pt);
      \draw[densely dotted,->] (1, 1) -- (2, 1);
      \draw[densely dotted,->] (1, 1) -- (1, 2);
      \node at (0.5,0) {$Q^*$};
    \end{tikzpicture}
    \caption{Illustration to Example~\ref{ex:r2}.}
    \label{fig:spin5}
  \end{figure}
\end{example}

\section{Polytopes with simplicial facets}
\label{sec:polyt-with-simpl}

The purpose of this section is to prove an auxiliary
result on polytopes (Proposition~\ref{prop:approx-simpl}),
which will be used in the proof of Theorem~\ref{theorem:le-2d}.

Let $V \cong \Q^n$ be a vector space of dimension $n$ and $Q \subseteq
V$ a full-dimensional polytope with $0 \in \topint(Q)$.

\begin{prop}
  \label{prop:approx-simpl}
  There exists a simplicial polytope $Q_s \subseteq V$ containing $Q$
  such that the simplicial facets of $Q$ are facets of $Q_s$.
\end{prop}

We denote the facets of $Q$ by $F_1, \ldots, F_r$. For every facet
$F_i$ we choose a hyperplane $H_i$ with normal vector $n_i \in V^*$,
i.\,e.~$H_i = \{ v \in V : \langle n_i, v \rangle = 1 \}$ such that
$F_i = Q \cap H_i$. Each hyperplane determines two open half-spaces
\begin{align*}
  H_i^- &\coloneqq \{ v \in V : \langle n_i, v \rangle < 1 \} \text{ and}\\
  H_i^+ &\coloneqq \{ v \in V : \langle n_i, v \rangle > 1 \}\text{,}
\end{align*}
such that $H_1^- \cap \ldots \cap H_r^- = \topint(Q)$.

\begin{definition}
  We say that $v \in V$ is \emph{beneath} (resp. is \emph{beyond})
  $F_i$ if $v$ belongs to $H_i^-$ (resp. to $H_i^+$).
\end{definition}

We will use the following result.

\begin{theorem}[{\cite[Theorem 5.2.1]{Gruenbaum:ConvexPolytopes}}]
  \label{thm:facial-structure}
  Let $v \in V$ such that $v$ is a vertex of $Q' \coloneqq \conv( \{ v
  \} \cup Q )$. Then
  \begin{enumerate}
  \item a face $F$ of $Q$ is a face of $Q'$ if and only if there
    exists a facet $E \preceq Q$ such that $F \subseteq E$ and $v$ is
    beneath $E$,
  \item if $F$ is a face of $Q$, then $F' \coloneqq \conv( \{ v \}
    \cup F )$ is a face of $Q'$ if and only if
    \begin{enumerate}
    \item either $v$ is contained in the affine span of $F$,
    \item or among the facets of $Q$ containing $F$ there is at least
      one such that $v$ is beneath it and at least one such that $v$
      is beyond it.
    \end{enumerate}
  \end{enumerate}
  Moreover, each face of $Q'$ is of one and only one of those types.
\end{theorem}

The following definitions are taken from \cite[Chapter~III, Sections 1
and 2]{Ewald:cc}.

\begin{definition}
  Let $\Ff$ be a fan in $V$ and $\sigma \in \Ff$. Then we set
  \begin{align*}
    \st( \sigma, \Ff ) & \coloneqq \{ \sigma' \in \Ff : \sigma
    \subseteq \sigma' \} && \text{(the \emph{star} of $\sigma$ in
      $\Ff$),}\\
    \clst( \sigma, \Ff ) & \coloneqq \{ \sigma'' \in \Ff : \sigma''
    \subseteq \sigma' \in \st( \sigma, \Ff ) \} && \text{(the
      \emph{closed star} of $\sigma$ in $\Ff$).}
  \end{align*}
\end{definition}

\begin{definition}
  Let $\sigma \subseteq V$ be a cone and $v \in V$ not contained in
  $\sigma$. Then we call $v \cdot \sigma \coloneqq \cone( \{ v \} \cup
  \sigma )$ the \emph{join} of $v$ and $\sigma$.
\end{definition}

\begin{definition}
  Let $\Ff$ be a fan in $V$ and $v \in V$. Assume that $v \cdot
  \sigma$ is defined for every $\sigma \in \Ff$ and that $\relint( v
  \cdot \sigma ) \cap \relint( v \cdot \sigma' ) = \emptyset$ whenever
  $\sigma, \sigma' \in \Ff$ are distinct. Then the fan $v \cdot \Ff
  \coloneqq \{ v \cdot \sigma : \sigma \in \Ff\}$ is called the
  \emph{join} of $v$ and $\Ff$.
\end{definition}

\begin{definition}
  Let $\Ff$ be a fan in $V$ and $v \in V$ a point such that there
  exists an (automatically uniquely determined) cone $\sigma \in \Ff$
  with $v \in \relint(\sigma)$.  Then we call the transition
  \begin{align*}
    \Ff \mapsto v \star \Ff \coloneqq ( \Ff \setminus \st( \sigma, \Ff
    ) ) \cup v \cdot ( \clst( \sigma, \Ff ) \setminus \st( \sigma, \Ff
    ) )
  \end{align*}
  the \emph{stellar subdivision} of $\Ff$ in direction of $v$.
\end{definition}

For a full-dimensional polytope $Q' \subseteq V$ with $0 \in
\topint(Q')$ we denote by $\Ff(Q')$ its face fan in $V$.

\begin{lemma}
  \label{lem:stellar-subdiv-on-polytope-side}
  Let $F \preceq Q$ be a non-simplicial face and $v \in \relint (F)$.
  Then there exists a polytope $Q' \subseteq V$ containing $Q$ such
  that
  \begin{align*}
    \Ff(Q') = v \star \Ff(Q)
  \end{align*}
  and the simplicial facets of $Q$ are facets of $Q'$.
\end{lemma}

\begin{proof}
  Let $F_{s_1}, \ldots, F_{s_{k}}$ be the facets not containing
  $F$. Choose $t > 1$ such that $\langle tv, n_{s_j} \rangle < 1$ for
  $j = 1, \ldots, k$ and set $v' \coloneqq tv$.  Then $v'$ is beneath
  the facets not containing $F$ and beyond the facets containing $F$.
  We set $Q' \coloneqq \conv( \{ v' \} \cup Q )$. As simplicial facets
  of $Q$ do not contain $F$, it follows from Theorem
  \ref{thm:facial-structure}(1) that the simplicial facets of $Q$ are
  facets of $Q'$.

  We now verify that $\Ff( Q' ) = v \star \Ff(Q)$.  It suffices to
  check that the sets of maximal cones coincide.

  By Theorem \ref{thm:facial-structure}(1), a facet $F_i \preceq Q$
  does not contain $F$ if and only if $F_i$ is a facet of
  $Q'$. Furthermore, the facets of $Q$ not containing $F$ are in
  correspondence with the maximal cones in $\Ff( Q ) \setminus \st(
  \Q_{\ge 0} F, \Ff( Q ) )$.

  Let $F'$ be a facet of $Q'$ which is not a facet of $Q$. By Theorem
  \ref{thm:facial-structure}(2), we have $F' = \conv( \{ v' \} \cup
  F'' )$ where $F'' \preceq Q$ is a face of codimension $2$ such that
  among the facets of $Q$ containing $F''$ there is at least one
  beneath and at least one beyond $v'$. Such faces $F''$ are in
  correspondence with the maximal cones in $\clst( \Q_{\ge 0} F, \Ff(
  Q ) ) \setminus \st( \Q_{\ge 0} F, \Ff( Q ) )$. The result follows
  from the equality $\Q_{\ge 0} F' = v \cdot \Q_{\ge 0} F''$.
\end{proof}

\begin{proof}[Proof of Proposition~\ref{prop:approx-simpl}]
  We can transform the fan $\Ff( Q )$ into a simplicial one by
  successively applying stellar subdivision to non-simplicial cones.
  By Lemma~\ref{lem:stellar-subdiv-on-polytope-side}, we also obtain a
  corresponding polytope.
\end{proof}

\section{Proof of Theorem~\ref{theorem:le-2d}: the inequality $\rho_X \le 2d$}
\label{sec:proof-ineq-theor}

Let $G/H \hookrightarrow X$ be a Gorenstein
spherical Fano embedding with
associated $G/H$-reflexive polytope $Q$.
The condition for $\Q$-factoriality from
Proposition~\ref{prop:sphfqf} can be
straightforwardly translated
into the setting of $G/H$-reflexive polytopes as follows:

\begin{prop}
$X$ is $\Q$-factorial if and only if every facet $\widehat{v}$ of $Q$
for $v \in V_{\Supp}(Q^*)$ has exactly $\rank X$ vertices in $V(Q)$,
where such a vertex can not be equal to $\rho(D)$ for more than one
$D \in \Dm$.
\end{prop}

Now assume that $X$ is $\Q$-factorial, of rank $r$, and of dimension $d$.
The proof of the
inequality $\rho_X \le 2d$ appearing
here is an extended version of the proof of the horospherical case
in \cite{Pasquier:FanoHorospherical}.
Note that, in contrast to the horospherical case, not all facets of
the polytope $Q$ are necessarily simplicial (only the facets dual to the
supported vertices of $Q^*$ are).

\begin{lemma}
  \label{lem:adjacency}
  Let $v \in V_{\Supp}( Q^* )$ and $u \in V( Q ) \cap \Nm$. If
  $\langle u,v \rangle = 0$, then there is a facet $F \preceq Q$
  containing $u$ and intersecting $\widehat{v}$ in a face of
  codimension $2$ of $Q$, \ie $u$ is adjacent to $\widehat{v}$.
\end{lemma}
\begin{proof}
  Let $e_1, \ldots, e_r$ be the vertices of the facet $\widehat{v}
  \preceq Q$. By \cite[Theorem 3.1.6]{Gruenbaum:ConvexPolytopes}, for
  all $j = 1, \ldots, r$ there is exactly one facet $F_j \preceq Q$
  containing $e_1, \ldots, e_{j-1}, e_{j+1}, \ldots, e_r$ being
  distinct from $\widehat{v}$. Let $u_j$ be a vertex of $F_j$ not
  contained in $\widehat{v} \cap F_j$. Then $(e_1, \ldots, e_{j-1},
  u_j, e_{j+1}, \ldots, e_r)$ is a basis of $\Nm_\Q$.

  Let $( e_1^*, \ldots, e_r^* )$ be the basis of $\Mm_\Q$ dual to $(
  e_1, \ldots,e_r )$. Let $j = 1, \ldots, r$.  Then $\langle e_j^*,
  u_j \rangle \neq 0$ (otherwise $u_j$ would be contained in the
  hyperplane spanned by $\{e_i\}_{ i\neq j}$). We define
  \begin{align*}
    \lambda_j \coloneqq \frac{-1 - \langle u_j, v \rangle }{ \langle
      u_j, e_j^* \rangle } \text{ and } v_j \coloneqq v + \lambda_j
    e_j^*\text{.}
  \end{align*}
  Then we have $\langle e_i, v_j \rangle = -1$ for $i \ne j$ and
  $\langle u_j, v_j \rangle = -1$. It follows that $v_j$ is a (not
  necessarily supported) vertex of $Q^*$ and $F_j = \widehat{v_j}$.
  We have $\lambda_j > 0$ as $-1 < \langle e_j, v_j \rangle$.  Then
  $\langle u, v_j \rangle = \lambda_j \langle u, e_j^* \rangle$ and
  hence
  \begin{align*}
    u \not\in F_j \text{ if and only if } \langle u, e_j^* \rangle \ge
    0.
  \end{align*}
  If $u \not \in F_j$ for all $j = 1, \ldots, r$, then $\langle u,
  e_j^* \rangle \ge 0$ for all $j = 1, \ldots, r$ and therefore $ u
  \in \cone ( e_1, \ldots, e_r )$, which implies $u = e_j \in F_j$ for
  some $j = 1, \ldots, r$, a contradiction. We obtain $u \in F_j$ for
  some $j = 1, \ldots, r$.
\end{proof}

\begin{cor}
  \label{sec:estimate_adjacent_lattice_points}
  Let $v \in V_{\Supp}(Q^*)$. Then we have $| V( \widehat{v} ) | = r$
  and
  \begin{align*}
    | \{ u \in V(Q) \cap \Vm \cap \Nm : \langle u, v \rangle = 0 \} |
    \le r.
  \end{align*}
\end{cor}
\begin{proof}
  As $\widehat{v}$ is a simplex, it has $r$ vertices, which we 
  denote by $u_1, \ldots, u_n$, and the first assertion follows.  For
  the second assertion, by Proposition~\ref{prop:approx-simpl},
  we may replace the polytope $Q$ by a
  simplicial polytope $Q'$ such that
  \begin{align*}
    V(Q) \cap \Vm \cap \Nm = V(Q') \cap \Vm \cap \Nm \text{ and
    }V_{\Supp}(Q^*) \subseteq V((Q')^*)\text{.}
  \end{align*}
  Let $u \in V( Q ) \cap \Vm \cap \Nm$ with
  $\langle u, v \rangle = 0$. By Lemma \ref{lem:adjacency}, there
  exists a facet $F_u \preceq Q$ adjacent to $\widehat{v}$ and
  containing $u$. As $Q$ is assumed simplicial, $F_u$ has vertices $u_1,
  \ldots, u_{i-1}, u, u_{i+1}, \ldots, u_r$ for some $i = 1, \ldots,
  r$. Since there are $r$ facets adjacent to $\widehat{v}$, the result
  follows.
\end{proof}

\begin{lemma}
  \label{lem:origin_in_interior_supp_verts_sph_roots}
  Let $Q \subseteq \Nm_\Q$ be a $G/H$-reflexive polytope. Then $0$ is
  contained in the interior of the convex hull of the supported
  vertices of $Q^*$ and $-\Sigma$.
\end{lemma}
\begin{proof}
  We set $\sigma \coloneqq \cone( V_{\Supp}( Q^* ) \cup ( -\Sigma ) )$
  and show $\sigma^\vee = 0$. Then $\sigma = ( \sigma^\vee )^\vee =
  0^\vee = \Mm_\Q$, and hence $0 \in \topint( \conv( V_{\Supp}( Q^* )
  \cup ( -\Sigma ) ) )$.

  Assume $0 \neq u \in \sigma^\vee$. Since $\langle u, -\Sigma \rangle
  \ge 0$, we have $u \in \Vm$. Then there is a supported facet $F
  \preceq Q$ and a rational number $t > 0$ such that $tu \in F$. Let
  $v$ be the supported vertex of $Q^*$ such that $\widehat{v} =
  F$. Then we have $0 \le \langle tu, v \rangle = -1$, a
  contradiction.
\end{proof}

It follows from
Lemma~\ref{lem:origin_in_interior_supp_verts_sph_roots} that there
exist positive natural numbers $m_v$ and $l_\gamma$ such that
\begin{align*}
  0 = \sum_{v \in V_{\Supp}(Q^*)}m_v v - \sum_{\gamma \in
    \Sigma}l_\gamma \gamma\text{.}
\end{align*}
We define $M \coloneqq \sum_{v \in V_{\Supp}(Q^*)}m_v$.

\begin{prop}
  \label{prop:bound-vert-of-q}
  We have
  \begin{align*}
    |V(Q) \cap \Vm \cap \Nm| \le 3r + \sum_{\gamma \in \Sigma} \sum_{u
      \in V(Q) \cap \Vm \cap \Nm} \frac{l_\gamma \langle u, \gamma
      \rangle}{M}\text{.}
  \end{align*}
\end{prop}
\begin{proof}
  For $u \in V( Q )\cap \Vm \cap \Nm$ we define
  \begin{align*}
    A( u ) \coloneqq \{ v \in V_{\Supp}(Q^*) : \langle u, v \rangle = -1 \}
    \text{ and } B( u ) \coloneqq \{ v\in V_{\Supp}(Q^*) : \langle u, v\rangle
    = 0\}.
  \end{align*}
  Then we have
  \begin{align*}
    0 & = \sum_{v\in V_{\Supp}(Q^*)} m_v \langle u,v \rangle -
    \sum_{\gamma \in
      \Sigma} l_\gamma \langle u,\gamma \rangle\\
    & \ge - \sum_{ v \in A( u ) } m_v + \sum_{ v \notin A( u ) \cup B(
      u ) } m_v - \sum_{\gamma \in \Sigma} l_\gamma \langle \gamma,
    u \rangle\\
    & = M - 2 \sum_{ v \in A( u ) } m_v - \sum_{ v \in B( u ) } m_v -
    \sum_{\gamma \in \Sigma} l_\gamma \langle u,\gamma \rangle
  \end{align*}
  and hence $M \le 2 \sum_{ v \in A( u ) } m_v + \sum_{ v \in B( u ) }
  m_v + \sum_{\gamma \in \Sigma} l_\gamma \langle u,\gamma \rangle$.
  Summing up this inequality over all $u \in V( Q ) \cap \Vm \cap
  \Nm$, we obtain
  \begin{align*}
    &|V(Q) \cap \Vm \cap \Nm| M \le \sum_u \sum_{ v \in A( u ) } 2 m_v
    + \sum_u \sum_{ v \in B( u ) } m_v + \sum_u \sum_{\gamma \in
      \Sigma} l_\gamma \langle
    u,\gamma \rangle\\
    = &\sum_{v\in V_{\Supp}(Q^*)} \sum_{ \langle u,v \rangle = -1 } 2
    m_v + \sum_{v\in V_{\Supp}(Q^*)} \sum_{ \langle u,v \rangle = 0 }
    m_v + \sum_{\gamma \in \Sigma} \sum_{u \in V(Q) \cap \Vm \cap \Nm}
    l_\gamma \langle u,\gamma \rangle.
  \end{align*}
  For a fixed $v$ the number of vertices $u$ of $V$ with $\langle u,v
  \rangle = -1$ is equal to the number of vertices of the facet
  $\widehat{v}$, which is exactly $r$.  By
  Corollary~\ref{sec:estimate_adjacent_lattice_points}, the number of
  $u \in V( Q ) \cap \Vm \cap \Nm$ with $\langle u,v \rangle = 0$ is
  less than or equal to $r$. Hence the result follows.
\end{proof}

\begin{prop}
  \label{prop:thineq}
  We have $\rho_X \le 2r + |\Dm| \le 2r + 2|S \setminus S^p| \le 2d$.
\end{prop}
\begin{proof}
  As $\langle u, \gamma \rangle \le 0$ for every $\gamma \in \Sigma$,
  $u \in \Vm$, it follows from
  Proposition~\ref{prop:bound-vert-of-q} that $| V(Q) \cap \Vm \cap
  \Nm | \le 3r$. Let $r'$ be the number of $G$-invariant prime
  divisors in $X$. It follows from \cite[Proposition~4.1.1]{brcox}
  that
  \begin{align*}
    \rho_X = r' + |\Dm| - r \le |V(Q) \cap \Vm \cap \Nm| + |\Dm| - r
    \le 2r + |\Dm|\text{.}
  \end{align*}
  By \cite[Theorem~30.22]{ti}, we obtain $|\Dm| \le 2|S\setminus
  S^p|$, and by \cite[Corollary~15.18]{ti}, we obtain $r +|S \setminus
  S^p| \le d$.
\end{proof}

\section{Proof of Theorem~\ref{theorem:le-2d}: the extreme case $\rho_X = 2d$}
\label{sec:proof-equal-theor}

Let $G/H \hookrightarrow X$ be a $\Q$-factorial Gorenstein
spherical Fano embedding of rank $r$ and of dimension $d$ with $\rho_X = 2d$.

\begin{prop}
  \label{prop:ssgp}
  We have $ |S \setminus S^p| = | R^+ \setminus R^+_{S^p}|$ where
  $R_{S^p}$ denotes the root system generated by $S^p$.
\end{prop}
\begin{proof}
  According to Proposition~\ref{prop:thineq}, we have $d = r + |S
  \setminus S^p|$.  On the other hand, we have $d = r + \dim G/P = r +
  | R^+ \setminus R^+_{S^p}|$.
\end{proof}

\begin{prop}
  \label{prop:ss-is-product}
  The root system $R$ is of type $A_1^k \times R_{S^p}$ for some $k
  \in \Z_{\ge 0}$.
\end{prop}
\begin{proof}
  As $S \subseteq R^+$ and $S^p \subseteq R^+_{S^p}$, it follows from
  Proposition~\ref{prop:ssgp} that $R^+ \setminus R^+_{S^p} = S
  \setminus S^p$, which means that every positive root not contained
  in $R^+_{S^p}$ is simple. Let $\alpha \in S \setminus S^p$. Then the
  irreducible factor of $R$ which contains $\alpha$ must be of type
  $A_1$ as it can not contain any non-simple positive root.
\end{proof}

As in Remark~\ref{rem:gssc}, we may assume
$G = G^{ss} \times C$ where $G^{ss}$ is semisimple simply-connected
and $C$ is a torus. According to Proposition~\ref{prop:ss-is-product},
we have $G^{ss} \cong \SL_2^k \times G'$ where $G'$ is the factor
corresponding to $R_{S^p}$.

\begin{cor}
  We may assume $G = \SL_2^k \times C$.
\end{cor}
\begin{proof}
  Let $P = L \ltimes P_u$ be the Levi decomposition with $T \subseteq
  L$. According to \cite[Theorem~4.7]{ti}, $[L, L]$ acts trivially on
  the open $B$-orbit in $X$. The result follows from the observation
  $G'= [L,L]$.
\end{proof}

\begin{cor}
  We have $V(Q) \cap \Vm \subseteq \Nm$ consists of all primitive ray
  generators where the rays correspond to $G$-invariant prime
  divisors.
\end{cor}

\begin{prop}
  \label{prop:no-sph-roots}
  We have $\Sigma = \emptyset$.
\end{prop}
\begin{proof}
  According to the proof of Proposition~\ref{prop:thineq}, we must
  have $| V(Q) \cap \Vm \cap \Nm | = 3r$. Taking into account
  Proposition~\ref{prop:bound-vert-of-q}, this implies
  \begin{align*}
    \sum_{\gamma \in \Sigma} \sum_{u \in V(Q) \cap \Vm \cap \Nm}
    \frac{l_\gamma \langle u, \gamma \rangle}{M} = 0\text{.}
  \end{align*}

  By Proposition \ref{prop:thineq} and \ref{prop:ss-is-product}, we
  obtain $|\Dm| = 2|S\setminus S^p| = 2k$ and thus all colors are of
  type $a$, $\Sigma \subset S$ and $\Dm(\alpha) \cap \Dm(\beta) =
  \emptyset$ for two distinct $\alpha,\beta \in S$. In particular,
  $\langle \rho(D), \alpha \rangle \le 1$ for all $D \in \Dm$ and all
  $\alpha \in \Sigma$ with equality if and only if $D \in
  \Dm(\alpha)$. As $\rho(D') + \rho(D'') = \alpha^\vee$ for
  $\Dm(\alpha)=\{ D', D'' \}$, we obtain $\langle \rho(D), \alpha
  \rangle = 0$ for all $D \not\in \Dm(\alpha)$. Since $X$ is complete,
  for every $\alpha \in \Sigma$ there exists a primitive ray generator $v
  \in \Vm$ such that $\langle v, \alpha \rangle < 0$.  As $l_\alpha >
  0$, we obtain $\Sigma = \emptyset$.
\end{proof}

By Proposition \ref{prop:no-sph-roots}, $X$ is horospherical, and therefore
the last part of Theorem~\ref{theorem:le-2d} follows from
\cite[Th\'{e}or\`{e}me 1.2]{Pasquier:FanoHorospherical}.

\addtocontents{toc}{\SkipTocEntry}
\section*{Acknowledgments}
We would like to thank our teacher Victor Batyrev for encouragement
and highly useful advice, as well as J\"{u}rgen Hausen for several
useful discussions. We are also grateful to Dmitry Timashev for
elaborating on Section~30.4 of his book. Finally, we thank
the referee for several helpful remarks and comments.

\bibliographystyle{amsalpha} \bibliography{gsfv}

\providecommand{\bysame}{\leavevmode\hbox to3em{\hrulefill}\thinspace}
\providecommand{\MR}{\relax\ifhmode\unskip\space\fi MR }
\providecommand{\MRhref}[2]{%
  \href{http://www.ams.org/mathscinet-getitem?mr=#1}{#2}
}
\providecommand{\href}[2]{#2}
\begin{thebibliography}{KKLV89}

\bibitem[AB04]{albr}
Valery~A. Alexeev and Michel Brion, \emph{Boundedness of spherical {F}ano
  varieties}, The {F}ano {C}onference, Univ. Torino, Turin, 2004, pp.~69--80.

\bibitem[Bat94]{Bat:DualPolyhedra}
Victor~V. Batyrev, \emph{Dual polyhedra and mirror symmetry for {C}alabi-{Y}au
  hypersurfaces in toric varieties}, J. Algebraic Geom. \textbf{3} (1994),
  no.~3, 493--535.

\bibitem[Bri89]{l-brion-pic}
Michel Brion, \emph{Groupe de {P}icard et nombres caract\'eristiques des
  vari\'et\'es sph\'eriques}, Duke Math. J. \textbf{58} (1989), no.~2,
  397--424.

\bibitem[Bri90]{brg}
\bysame, \emph{Vers une g\'en\'eralisation des espaces sym\'etriques}, J.
  Algebra \textbf{134} (1990), no.~1, 115--143.

\bibitem[Bri93]{sphmori}
\bysame, \emph{Vari\'et\'es sph\'eriques et th\'eorie de {M}ori}, Duke Math. J.
  \textbf{72} (1993), no.~2, 369--404.

\bibitem[Bri97]{Brion:cc}
M.~Brion, \emph{{Curves and divisors in spherical varieties}}, {Algebraic
  groups and {L}ie groups}, {Austral. Math. Soc. Lect. Ser.}, vol.~9, Cambridge
  Univ. Press, Cambridge, 1997, pp.~21--34.

\bibitem[Bri07]{brcox}
Michel Brion, \emph{{The total coordinate ring of a wonderful variety}}, J.
  Algebra \textbf{313} (2007), no.~1, 61--99.

\bibitem[Cas06]{Casagrande:numVerts}
Cinzia Casagrande, \emph{The number of vertices of a {F}ano polytope}, Ann.
  Inst. Fourier (Grenoble) \textbf{56} (2006), no.~1, 121--130.

\bibitem[Deb03]{Debarre:FanoVarieties}
Olivier Debarre, \emph{Fano varieties}, Higher dimensional varieties and
  rational points ({B}udapest, 2001), Bolyai Soc. Math. Stud., vol.~12,
  Springer, Berlin, 2003, pp.~93--132.

\bibitem[Dem68]{Demazure:Bott}
Michel Demazure, \emph{Une d\'emonstration alg\'ebrique d'un th\'eor\`eme de
  {B}ott}, Invent. Math. \textbf{5} (1968), 349--356.

\bibitem[Dem76]{Demazure:BottSimple}
\bysame, \emph{A very simple proof of {B}ott's theorem}, Invent. Math.
  \textbf{33} (1976), no.~3, 271--272.

\bibitem[Ewa96]{Ewald:cc}
G{\"u}nter Ewald, \emph{Combinatorial convexity and algebraic geometry},
  Graduate Texts in Mathematics, vol. 168, Springer-Verlag, New York, 1996.

\bibitem[Fos98]{foschi}
A.~Foschi, \emph{{Vari{\'e}t{\'e}s magnifiques et polytopes moment}}, Ph.D.
  thesis, Universit{\'e} de Grenoble I, 1998.

\bibitem[Ful93]{fulttor}
William Fulton, \emph{Introduction to toric varieties}, Annals of Mathematics
  Studies, vol. 131, Princeton University Press, Princeton, NJ, 1993, The
  William H. Roever Lectures in Geometry.

\bibitem[GH15]{gh14}
Giuliano Gagliardi and Johannes Hofscheier, \emph{Homogeneous spherical data of
  orbits in spherical embeddings}, Transform. Groups, electronically published
  on January 18, 2015, DOI: http://dx.doi.org/10.1007/s00031-014-9297-2 (to
  appear in print).

\bibitem[Gr{\"u}03]{Gruenbaum:ConvexPolytopes}
Branko Gr{\"u}nbaum, \emph{Convex polytopes}, second ed., Graduate Texts in
  Mathematics, vol. 221, Springer-Verlag, New York, 2003, Prepared and with a
  preface by Volker Kaibel, Victor Klee and G{\"u}nter M. Ziegler.

\bibitem[HUL01]{FundamentalsConvexAnalysis}
Jean-Baptiste Hiriart-Urruty and Claude Lemar{\'e}chal, \emph{Fundamentals of
  convex analysis}, Grundlehren Text Editions, Springer-Verlag, Berlin, 2001.

\bibitem[Hum78]{Humphreys:LieAlgebras}
James~E. Humphreys, \emph{Introduction to {L}ie algebras and representation
  theory}, Graduate Texts in Mathematics, vol.~9, Springer-Verlag, New York,
  1978, Second printing, revised.

\bibitem[KKLV89]{KKLV}
Friedrich Knop, Hanspeter Kraft, Domingo Luna, and Thierry Vust, \emph{Local
  properties of algebraic group actions}, Algebraische {T}ransformationsgruppen
  und {I}nvariantentheorie, DMV Sem., vol.~13, Birkh\"auser, Basel, 1989,
  pp.~63--75.

\bibitem[KKV89]{knoppic}
Friedrich Knop, Hanspeter Kraft, and Thierry Vust, \emph{{The {P}icard group of
  a {$G$}-variety}}, {Algebraische {T}ransformationsgruppen und
  {I}nvariantentheorie}, {DMV Sem.}, vol.~13, Birkh{\"a}user, Basel, 1989,
  pp.~77--87.

\bibitem[Kno91]{knopsph}
Friedrich Knop, \emph{{The {L}una-{V}ust theory of spherical embeddings}},
  {Proceedings of the {H}yderabad {C}onference on {A}lgebraic {G}roups
  ({H}yderabad, 1989)} (Madras), Manoj Prakashan, 1991, pp.~225--249.

\bibitem[Kno95]{Knop:Borbits}
\bysame, \emph{{On the set of orbits for a {B}orel subgroup}}, Comment. Math.
  Helv. \textbf{70} (1995), no.~2, 285--309.

\bibitem[Kno14]{Knop:loc}
\bysame, \emph{Localization of spherical varieties}, Algebra Number Theory
  \textbf{8} (2014), no.~3, 703--728.

\bibitem[Lun97]{Luna:cc}
D.~Luna, \emph{{Grosses cellules pour les vari{\'e}t{\'e}s sph{\'e}riques}},
  {Algebraic groups and {L}ie groups}, {Austral. Math. Soc. Lect. Ser.},
  vol.~9, Cambridge Univ. Press, Cambridge, 1997, pp.~267--280.

\bibitem[Lun01]{Luna:typea}
\bysame, \emph{{Vari{\'e}t{\'e}s sph{\'e}riques de type {$A$}}}, Publ. Math.
  Inst. Hautes {\'E}tudes Sci. (2001), no.~94, 161--226.

\bibitem[LV83]{lunavust}
D.~Luna and Th. Vust, \emph{{Plongements d'espaces homog{\`e}nes}}, Comment.
  Math. Helv. \textbf{58} (1983), no.~2, 186--245.

\bibitem[MFK94]{MumfordFogartyKirwan:GIT}
David Mumford, John Fogarty, and Frances~Clare Kirwan, \emph{{Geometric
  invariant theory}}, third ed., {Ergebnisse der Mathematik und ihrer
  Grenzgebiete (2) [Results in Mathematics and Related Areas (2)]}, vol.~34,
  Springer-Verlag, Berlin, 1994.

\bibitem[Pas08]{Pasquier:FanoHorospherical}
Boris Pasquier, \emph{{Vari{\'e}t{\'e}s horosph{\'e}riques de {F}ano}}, Bull.
  Soc. Math. France \textbf{136} (2008), no.~2, 195--225.

\bibitem[Ros63]{ros63}
Maxwell Rosenlicht, \emph{Questions of rationality for solvable algebraic
  groups over nonperfect fields}, Ann. Mat. Pura Appl. (4) \textbf{61} (1963),
  97--120.

\bibitem[Spr09]{Springer:LAG}
Tonny~Albert Springer, \emph{{Linear algebraic groups}}, second ed., {Modern
  Birkh{\"a}user Classics}, Birkh{\"a}user Boston Inc., Boston, MA, 2009.

\bibitem[Tim11]{ti}
Dmitry~A. Timashev, \emph{{Homogeneous spaces and equivariant embeddings}},
  {Encyclopaedia of Mathematical Sciences}, vol. 138, Springer, Heidelberg,
  2011, Invariant Theory and Algebraic Transformation Groups, 8.

\bibitem[Was96]{r2}
B.~Wasserman, \emph{Wonderful varieties of rank two}, Transform. Groups
  \textbf{1} (1996), no.~4, 375--403.

\end{thebibliography}

\end{document}